\documentclass[10pt,reqno]{amsart}
\usepackage{amsmath,amsfonts,amsthm,amscd,amssymb,graphicx}
\numberwithin{equation}{section}

\usepackage{ifthen} 

\provideboolean{shownotes} 
\setboolean{shownotes}{true} 
\newcommand{\margnote}[1]{
\ifthenelse{\boolean{shownotes}}%
{\marginpar{\raggedright\tiny\texttt{#1}}}%
{}%
}

\theoremstyle{plain}
\newtheorem{theorem}{Theorem}[section]

\newtheorem{lemma}[theorem]{Lemma}

\newtheorem{proposition}[theorem]{Proposition}
\newtheorem{prop}[theorem]{Proposition}
\newtheorem{corollary}[theorem]{Corollary}

\theoremstyle{definition}

\newtheorem{remark}[theorem]{Remark}

\newcommand{\R}{\mathbb{R}}

\newcommand{\C}{\mathbb{C}}
\newcommand{\A}{\mathbb{A}}

\newcommand{\Real}{\textrm{\rm Re}\,}
\newcommand{\Imag}{\textrm{\rm Im}\,}
\renewcommand{\Re}{\Real}
\renewcommand{\Im}{\Imag}

\newcommand{\cO}{\mathcal{O}}
\newcommand{\cG}{\mathcal{G}}
\newcommand{\cK}{\mathcal{K}}
\newcommand{\cS}{\mathcal{S}}
\newcommand{\cL}{\mathcal{L}}
\newcommand{\cJ}{\mathcal{J}}

\newcommand{\cV}{\mathcal{V}}
\newcommand{\cF}{\mathcal{F}}
\newcommand{\cU}{\mathcal{U}}

\newcommand{\bA}{\bar A}
\newcommand{\bL}{\bar L}

\newcommand{\half}{\tfrac{1}{2}}

\newcommand{\iprod}[1]{\langle{#1}\rangle}
\newcommand{\dt}{{\frac{d}{dt}}}

\newcommand{\Dxk}{{\partial^k_x}}
\newcommand{\cE}{\mathcal{E}}

\newcommand{\RR}{\mathbb{R}}

\newcommand{\hole}[1]{
\ifthenelse{\boolean{shownotes}}%
{\begin{center} \fbox{ \rule {.25cm}{0cm}
\rule[-.1cm]{0cm}{.4cm} \parbox{.85\textwidth}{\begin{center}
\texttt{#1}\end{center}} \rule {.25cm}{0cm}}\end{center}}
{}
}

\begin{document}

\title[Stability of radiative shock profiles]{Stability of radiative shock profiles for hyperbolic-elliptic coupled systems}

\author[T. Nguyen]{Toan Nguyen}

\address{Department of Mathematics\\Indiana University\\Bloomington,
IN 47405 (U.S.A.)}

\email{nguyentt@indiana.edu}

\author[R. G. Plaza]{Ram\'on G. Plaza}

\address{Departamento de Matem\'aticas y
Mec\'anica\\IIMAS-UNAM\\Apdo. Postal 20-726, C.P. 01000 M\'exico
D.F. (M\'exico)}

\email{plaza@mym.iimas.unam.mx}

\author[K. Zumbrun]{Kevin Zumbrun}

\address{Department of Mathematics\\Indiana University\\Bloomington,
IN 47405 (U.S.A.)}

\email{kzumbrun@indiana.edu}

\date{\today}

\thanks{The research of TN and KZ was supported in part by the National
Science Foundation, award number DMS-0300487.
The research of RGP was partially supported by DGAPA-UNAM through the
program PAPIIT, grant IN-109008.
RGP is warmly grateful to the
Department of Mathematics, Indiana University, for their hospitality and
financial support during two short visits in May 2008 and April 2009,
when this research  was carried out.
TN, RGP, and KZ are warmly grateful to Corrado Lattanzio and Corrado Mascia for their interest in this work and for many helpful conversations,
as well as their collaboration in concurrent work on the scalar case.
}

\begin{abstract}
Extending previous work with Lattanzio and Mascia on the scalar
(in fluid-dynamical variables) Hamer model for a radiative gas,
we show nonlinear orbital asymptotic stability of small-amplitude
shock profiles of general systems of coupled hyperbolic--eliptic
equations of the type modeling a radiative gas,
that is, systems of conservation laws coupled with an elliptic equation
for the radiation flux,
including in particular the standard Euler--Poisson model for
a radiating gas.
The method is based on the derivation of pointwise Green function bounds and
description of  the linearized  solution operator,
with the main difficulty being the construction of
the resolvent kernel in the case of an eigenvalue system of
equations of degenerate type.
Nonlinear stability then follows in standard fashion by linear
estimates derived from these pointwise bounds, combined with
nonlinear-damping type energy estimates.
\end{abstract}

\maketitle


\setcounter{tocdepth}{1}



\section{Introduction}\label{sec:intro}

In the theory of non-equilibrium radiative hydrodynamics, it is often assumed that an inviscid compressible fluid interacts with radiation through energy exchanges.
One widely accepted model \cite{VK} considers the one dimensional Euler
system of equations coupled with an elliptic equation for the radiative
energy,
 or {\it Euler--Poisson equation}.
With this system in mind, this paper considers \textit{general hyperbolic-elliptic coupled systems} of the form,
\begin{equation}
\begin{aligned}
u_{t}+ f(u)_{x} + L q_{x} &= 0,\\
-q_{xx} + q + g(u)_{x} &=0,
\end{aligned}
\label{eq:systemold}
\end{equation}
with $(x,t) \in \R \times [0,+\infty)$ denoting space and time, respectively, and where the unknowns $u\in \cU \subseteq \R^n$, $n \geq 1$, play the role of state variables,
whereas
$q\in \R$ represents a general heat flux. In addition, $L \in \R^{n \times 1}$ is a constant vector, and $f \in C^2(\cU ; \R^n)$ and $g \in C^2(\cU;\R)$ are nonlinear vector- and scalar-valued flux functions, respectively.

The study of general systems like \eqref{eq:systemold} has been the subject of active research in recent years \cite{KaNN1,KaNN2,KN4,LMS1}. There
exist,
however, more complete results regarding the simplified model of a radiating gas, also known as the \textit{Hamer model} \cite{Hm}, consisting of a scalar velocity equation (usually endowed with a Burgers' flux function which approximates the Euler system), coupled with a scalar elliptic equation for the heat flux. Following the authors' concurrent analysis
with Lattanzio and Mascia of the
reduced scalar model \cite{LMNPZ1}, this work studies the asymptotic stability of \textit{general radiative shock profiles}, which are traveling wave solutions to system \eqref{eq:systemold} of the form
\begin{equation}
\label{tws}
u(x,t) = U(x - st), \qquad q(x,t) = Q(x-st),
\end{equation}
with asymptotic limits
\[
U(x) \to u_\pm, \qquad Q(x) \to 0, \qquad \text{as} \; \; x \to \pm \infty,
\]
being $u_\pm \in \cU \subseteq \R^n$ constant states and $s \in \R$ the shock speed. The main assumption is that the triple $(u_+,u_-,s)$ constitutes a shock front \cite{La1} for the underlying ``inviscid" system of conservation laws
\begin{equation}
\label{hscl}
u_t + f(u)_x = 0,
\end{equation}
satisfying canonical jump conditions of Rankine-Hugoniot type,
\begin{equation}
\label{RH}
f(u_+) - f(u_-) - s(u_+ - u_-) = 0,
\end{equation}
plus classical Lax entropy conditions. In the sequel we denote the jacobians of the nonlinear flux functions as
\[
A(u) := Df(u) \in \R^{n \times n}, \qquad B(u) := Dg(u) \in \R^{1 \times n}, \qquad u \in \cU.
\]
Right and left eigenvectors of $A$ will be denoted as $r \in \R^{n \times 1}$ and $l \in \R^{1 \times n}$, and we suppose that system \eqref{hscl} is hyperbolic, so that $A$ has real eigenvalues
$a_1\leq \dots \leq a_n$.

It is assumed that system \eqref{eq:systemold} represents some sort of regularization of the inviscid system \eqref{hscl} in the following sense. Formally, if we eliminate the $q$ variable, then we end up with a system of form
\[
u_t + f(u)_x = (LB(u)u_x)_x + (u_t + f(u)_x)_{xx},
\]
which requires a nondegeneracy hypothesis
\begin{equation}
\label{eq:mainassump}
l_p \cdot (B \otimes L^\top r_p) > 0,
\end{equation}
for some $1 \leq p \leq n$, in order to provide a good dissipation term along the $p$-th characteristic field in its Chapman-Enskog expansion \cite{ST}.

\smallskip

More precisely, we make the following structural assumptions:

\begin{equation}
\label{S0}
f, g \in C^2 \quad \text{(regularity)},
\tag{S0}
\end{equation}

\begin{equation}
\label{S1}
\tag{S1}
\begin{minipage}[c]{4.5in}
For all $u \in \cU$ there exists $A_0$ symmetric, positive definite such that
$A_0 A$ is symmetric, and $A_0LB$ is symmetric, positive semi-definite of rank one (symmetric dissipativity $\Rightarrow$ non-strict hyperbolicity). Moreover, we assume that the principal eigenvalue $a_p$ of $A$ is simple.
\end{minipage}
\end{equation}

\begin{equation}
\label{S2}
\tag{S2}
\begin{minipage}[c]{4.5in}
No eigenvector of $A$ lies in $\ker LB$ (genuine coupling).
\end{minipage}
\end{equation}

\begin{remark}
Assumption \eqref{S1} assures non-strict hyperbolicty of the system, with simple principal characteristic field. Notice that \eqref{S1} also implies that $(A_0)^{1/2}A(A_0)^{-1/2}$ is symmetric, with real and semi-simple spectrum, and that, likewise, $(A_0)^{1/2}B(A_0)^{-1/2}$ preserves symmetric positive semi-definiteness with rank one. Assumption \eqref{S2} defines a general class of hyperbolic-elliptic equations analogous to the class defined by Kawashima and Shizuta \cite{KaTh,KaSh1,ShKa1} and compatible with \eqref{eq:mainassump}. Moreover, there is an equivalent condition to \eqref{S2} given by the following
\begin{lemma}[Shizuta--Kawashima \cite{KaSh1,ShKa1}]
\label{lemmaKaw}
Under \eqref{S0} - \eqref{S1}, assumption \eqref{S2} is equivalent to the existence of a skew-symmetric matrix valued function $K : \cU \to \R^{n \times n}$ such that
\begin{equation}
\label{KALB}
\Re (KA + A_0LB) > 0,
\end{equation}
for all $u \in \cU$.
\end{lemma}
\begin{proof}
See, e.g., \cite{HuZ1}.
\end{proof}
\end{remark}

As usual, we can reduce the problem to the analysis of a stationary profile with $s=0$, by introducing a convenient change
of variable and relabeling the flux function $f$ accordingly.
Therefore, we end up with a stationary solution $(U,Q)(x)$ of the
system
\begin{equation}
\label{eq:profileeqn}
\begin{aligned}
f(U)_x + LQ_x &= 0,\\
-Q_{xx} + Q + g(U)_x &= 0.
\end{aligned}
\end{equation}

After such normalizations and under \eqref{S0} - \eqref{S2}, we make the following assumptions about the shock:

\begin{equation}
\label{H0}
f(u_+) = f(u_-), \quad  \text{(Rankine-Hugoniot jump conditions)},
\tag{H0}
\end{equation}

\begin{equation}
\label{H1}
\tag{H1}
\begin{aligned}
a_p(u_+) < \; &0 < a_{p+1}(u_+), \\
a_{p-1}(u_-) <\; &0 < a_p (u_-),
\end{aligned}
\quad \;\,\text{(Lax entropy conditions)},
\end{equation}

\begin{equation}
\label{H2}
\tag{H2}
(\nabla a_p)^\top r_p \neq 0, \; \; \text{for all }\, u \in \cU, \quad \,\text{(genuine nonlinearity)},
\end{equation}

\begin{equation}
\label{H3}
\tag{H3}
l_p(u_\pm) L B(u_\pm) r_p(u_\pm) > 0, \quad\qquad \;\;\;\;\,\text{(positive diffusion)}.
\end{equation}

\begin{remark}
Systems of form \eqref{eq:systemold} arise in the study of radiative hydrodynamics, for which the paradigmatic system has the form
\begin{equation}
\label{rEuler}
\begin{aligned}
 \rho_t + (\rho u)_x &= 0,\\
(\rho u)_t + (\rho u^2 + p)_x &= 0,\\
\Big(\rho(e + \half u^2)\Big)_t + \Big(\rho u( e + \half u^2) +p u + q \Big)_x &= 0, \\
-q_{xx} + aq + b(\theta^4)_x &= 0,
\end{aligned}
\end{equation}
which corresponds to the one dimensional Euler system coupled with an elliptic equation describing radiations in a stationary diffusion regime.
In \eqref{rEuler}, $u$ is the velocity of the fluid, $\rho$ is the mass density and $\theta$ denotes the temperature. Likewise, $p = p(\rho,\theta)$ is the pressure and $e = e(\rho,\theta)$ is the internal energy. Both $p$ and $e$ are assumed to be smooth functions of $\rho > 0$, $\theta > 0$ satisfying
\[
p_\rho > 0, \quad p_\theta \neq 0, \quad e_\theta > 0.
\]
Finally, $q = \rho \chi_x$ is the radiative heat flux, where $\chi$ represents the radiative energy, and $a, b > 0$ are positive constants related to absorption. System \eqref{rEuler} can be (formally) derived from a more complete system involving a kinetic equation for the specific intensity of radiation. For this derivation and further physical considerations on \eqref{rEuler} the reader is referred to \cite{VK,LCG1,KaNN2}.
\end{remark}

The existence and regularity of traveling wave type solutions of
(\ref{eq:systemold}) under hypotheses \eqref{S0} - \eqref{S2}, \eqref{H0} - \eqref{H3} is known
, even in the more general case of non-convex velocity fluxes (assumption \eqref{H2} does not hold). For details of existence, as well as further properties of the profiles such as monotonicity and regularity under small-amplitude assumption (features which will be used throughout the analysis), the reader is referred to \cite{LMS1,LMS2}.

\subsection{Main results} In the spirit
of \cite{ZH,MaZ1,MaZ4,MaZ5}, we first consider the
linearized equations of \eqref{eq:systemold} about the profile
$(U,Q)$:
\begin{equation}
    \begin{aligned}
\begin{matrix}
    u_{t}+ (
A
(U)u)_{x} +L q_{x}=0, \\
    - q_{xx} +q +(
B
(U)u)_{x} =0,\end{matrix}
    \end{aligned}
     \label{eq:lin}
\end{equation}
with initial data $u(0)=u_0$. Hence, the Laplace transform applied to system \eqref{eq:lin} gives
\begin{equation}
    \begin{aligned}
    \lambda u+ (A(U)\,u)_x +L q_x&= S, \\
    - q_{xx} +q +(B(U)u)_x &=0,
    \end{aligned}
     \label{eq:linLaplace}
\end{equation}
where source $S$ is the initial data $u_0$.
As it is customary in related nonlinear wave stability analyses
(see, e.g., \cite{AGJ,San,ZH,Z3}), we start by studying the underlying spectral
problem, namely, the homogeneous version of system \eqref{eq:linLaplace}:
\begin{equation}
\label{spectralsyst}
    \begin{aligned}
    \lambda u+ (A(U)\,u)_x +L q_x&= 0, \\
    - q_{xx} +q +(B(U)u)_x &=0.
    \end{aligned}
\end{equation}

An evident necessary condition for orbital stability is the absence of $L^2$ solutions
to \eqref{spectralsyst} for values of $\lambda$ in $\{ \Real \lambda \geq 0\} \backslash \{0\}$,
being $\lambda = 0$ the eigenvalue associated to translation invariance.
This spectral stability condition can be expressed in terms of the \textit{Evans function},
an analytic function playing a role for differential operators analogous to that played by the
characteristic polynomial for finite-dimensional operators (see \cite{AGJ,San,GZ,ZH,MaZ1}
and the references therein). The main property of the Evans function is that, on the resolvent set
of a certain operator ${\mathcal L}$, its zeroes coincide in both location and multiplicity with
the eigenvalues of ${\mathcal L}$. Thence, we express the spectral stability condition as follows:

\begin{equation}
\label{SS}
\tag{SS}
\begin{minipage}[c]{4.5in}
There exists no zero of the Evans function $D$ on $\{\Re \lambda \ge 0\}\setminus \{0\}$;
equivalently, there exist no nonzero eigenvalues of $\mathcal{L}$ with $\Re \lambda\ge 0$.
\end{minipage}
\end{equation}

\medskip

\noindent Like in previous analyses \cite{ZH,Z3,Z7}, we define the following {\it stability condition} (or {\it Evans function condition}) as follows:
%
%
%
%
%

\begin{equation}
\label{D}
\tag{D}
\begin{minipage}[c]{4.5in}
There exists precisely one zero (necessarily at $\lambda=0$; see Lemmas \ref{lem-Evansfns} - \ref{lemma-mD})
of the Evans function on the nonstable half plane $\{\Re \lambda \ge 0\}$,
\end{minipage}
\end{equation}

\smallskip

\noindent which implies the spectral stability condition \eqref{SS} plus the condition that $D$ vanishes at $\lambda = 0$ at order one. Notice that just like in the scalar case \cite{LMNPZ1}, due to the degenerate nature of system \eqref{spectralsyst}
(observe that $A(U)$ vanishes at $x=0$) the number of decaying modes at $\pm \infty$, spanning
possible eigenfunctions, depends on the region of space around the singularity. Therefore the definition of $D$ is given in terms of the Evans functions $D_\pm$ in regions $x\gtrless 0$, with same regularity and spectral properties (see its definition in \eqref{rEvans} and Lemmas \ref{lem-Evansfns} - \ref{lemma-mD} below).


%
%
%

\medskip


Our main result is then as follows.

\begin{theorem}\label{theo-main} Assuming
\eqref{eq:mainassump}, \eqref{S0}--\eqref{S2}, \eqref{H0}--\eqref{H3},
and the spectral stability
condition
\eqref{D},
then the Lax radiative shock profile $(U,Q)$ with sufficiently small amplitude is
asymptotically orbitally stable. More precisely, the solution
$(\tilde u,\tilde q)$ of \eqref{eq:systemold} with initial data
$\tilde u_0$ satisfies
\begin{equation}\begin{aligned}
&|\tilde u(x,t) - U(x-\alpha(t))|_{L^p} \le C(1+t)^{-\frac
12(1-1/p)}|u_0|_{L^1\cap H^4}\\&|\tilde u(x,t) -
U(x-\alpha(t))|_{H^4} \le C(1+t)^{-1/4}|u_0|_{L^1\cap H^4}
\end{aligned}\end{equation}
and
\begin{equation}\begin{aligned}
&|\tilde q(x,t) - Q(x-\alpha(t))|_{W^{1,p}} \le C(1+t)^{-\frac
12(1-1/p)}|u_0|_{L^1\cap H^4}\\&|\tilde q(x,t) -
Q(x-\alpha(t))|_{H^5} \le C(1+t)^{-1/4}|u_0|_{L^1\cap H^4}
\end{aligned}\end{equation}
for initial perturbation $u_0:=\tilde u_0 - U$ that are sufficiently
small in $L^1\cap H^4$, for all $p\ge 2$, for some $\alpha(t)$
satisfying $\alpha(0)=0$ and
\begin{equation}\begin{aligned}
&|\alpha(t)|\le C|u_0|_{L^1\cap H^4}
\\&|\dot\alpha(t)|\le C(1+t)^{-1/2}|u_0|_{L^1\cap H^4}.
\end{aligned}\end{equation}

\end{theorem}

\begin{remark} \textup{The time-decay rate of $q$ is not optimal. In fact, it can be improved
as we observe that $|q(t)|_{L^2} \le C|u_x(t)|_{L^2}$ and $|u_x(t)|_{L^2}$ is expected to decay
like $t^{-1/2}$; however, we omit the detail of carrying this out.
Likewise, assuming in addition
a small $L^1$ first moment on the initial perturbation,
we could obtain by the approach of \cite{Ra} the sharpened
bounds $|\dot \alpha|\le C(1+t)^{\sigma-1}$,
and $| \alpha -\alpha(+\infty)|\le C(1+t)^{\sigma-1/2}$,
for $\sigma>0$ arbitrary, including in particular the information
that $\alpha$ converges to a specific limit (phase-asymptotic
orbital stability); however, we omit this again in favor of simplicity.}
\end{remark}

We shall prove the following result in the appendix, verifying Evans
condition (D).

\begin{theorem}
\label{spectralstability}
For $\epsilon := |u_+ - u_-|$ sufficiently small, radiative shock profiles are spectrally stable.
\end{theorem}
\begin{corollary}\label{prop-verifyD} The condition (D) is satisfied for small amplitudes.
\end{corollary}
\begin{proof}
 In Lemmas \ref{lem-Evansfns} - \ref{lemma-mD} below, we show that $D(\lambda)$ has a single zero at $\lambda=0$.
Together with Theorem \ref{spectralstability}, this gives the result.
\end{proof}

\subsubsection{Discussion}
Prior to \cite{LMNPZ1},
asymptotic stability  of radiative shock profiles has been
studied in the scalar case in \cite{KN1}
for the particular case of Burgers velocity flux
and for linear
$g(u) = M u$, with constant $M$.
Another
scalar result is the partial analysis of
Serre \cite{Ser7} for the exact  Rosenau model.
In the case of systems, we mention the stability
result of \cite{LCG2} for the full  Euler radiating system under
special {\it zero-mass} perturbations, based on
an adaptation of the classical  energy method of Goodman-Matsumura-Nishihara \cite{Go1,MN}.
Here, we recover for systems, under
general (not necessarily zero-mass) perturbations,
the sharp rates of decay established in
\cite{KN1} for the scalar case.

We mention that works \cite{KN1,LMNPZ1} in the scalar case
concerned also {\it large-amplitude} shock profiles
(under the Evans condition (D), automatically satisfied
in the Burgers case \cite{KN1}).
At the expense of further effort book-keeping-- specifically in
the resolution of flow near the singular point and construction
of the resolvent-- we could obtain by our methods a large-amplitude
result similar to that of \cite{LMNPZ1}.
However, we greatly simplify the exposition by the small-amplitude
assumption allowing us to approximately diagonalize {\it before}
carrying out these steps.
As the existence theory is only for small-amplitude shocks, with
upper bounds on the amplitudes for which existence holds, known
to occur, and since the domain of our hypotheses in \cite{LMNPZ1}
does not cover the whole domain of existence in the scalar
case (in contrast to \cite{KN1}, which does address the entire domain
of existence), we have chosen here for clarity to restrict to
the small-amplitude setting.
It would be interesting to carry out a large-amplitude analysis valid
on the whole domain of existence in the system case.

\subsection{Abstract framework}\label{subsec:framework}
Before beginning the analysis, we orient ourselves with a few simple
observations framing the problem in a more standard way. Consider
now the inhomogeneous version
\begin{equation}
    \begin{aligned}
    u_{t}+ (A(U)\,u)_{x} +L q_{x}&=g,\\
    - q_{xx} +q +(B(U)\, u)_{x} &=h,
    \end{aligned}
     \label{eq:fulllin}
\end{equation}
of \eqref{eq:lin}, with initial data $u(x,0)=u_0$. Defining the
compact operator $\cK:=(-\partial_x^2+ 1)^{-1}$ of order $-1$, and
the bounded operator
$$
\cJ:=  \partial_x L \cK \partial_x B(U)
$$
of order $0$, we may rewrite this as a nonlocal equation
\begin{equation}
\begin{aligned}
    u_{t}+ (A(U)\,u)_{x} + \cJ u&=
\partial_x L \cK h  + g,\\
u(x,0)&=u_0(x)
\end{aligned}
     \label{eq:redlin}
\end{equation}
in $u$ alone, recovering $q$ by
\begin{equation}
    q=-\cK\partial_x B(U) u +\cK h.
     \label{eq:qrec}
\end{equation}
The generator
$\cL:= - (A(U)\,u)_{x} - \cJ u$
of \eqref{eq:redlin}
is a zero-order perturbation of the generator
$- A(U)u_x$
of a
hyperbolic equation, so generates a $C^0$ semigroup $e^{\cL t}$ and
an associated Green distribution $G(x,t;y):=e^{\cL t}\delta_y(x)$.
Moreover, $e^{\cL t}$ and $G$ may be expressed through the inverse
Laplace transform formulae
\begin{equation}\label{iLT}
\begin{aligned}
e^{\cL t}&=\frac{1}{2\pi i} \int_{\eta -i\infty}^{\eta+i\infty}
e^{\lambda t} (\lambda-\cL)^{-1}d\lambda,\\
G(x,t;y)&=\frac{1}{2\pi i} \int_{\eta -i\infty}^{\eta+i\infty}
e^{\lambda t} G_\lambda(x,y)d\lambda,\\
\end{aligned}
\end{equation}
for all $\eta\ge \eta_0$, where
$G_\lambda(x,y):=(\lambda-\cL)^{-1}\delta_y(x)$ is the resolvent
kernel of $\cL$.

Collecting information, we may write the solution of
\eqref{eq:fulllin} using Duhamel's principle/variation of constants
as
\begin{equation}\label{eq:drep}
\begin{aligned}
u(x,t)&= \int_{-\infty}^{+\infty} G(x,t;y)u_0(y)dy \\
&\quad + \int_0^t \int_{-\infty}^{+\infty} G(x,t-s;y)
(\partial_x L \cK h  + g)(y,s)\, dy\, ds,\\
q(x,t)& = \Big((-\cK \partial_x B(U))u +\cK h \Big) (x,t),
\end{aligned}
\end{equation}
where $G$ is determined through \eqref{iLT}.


That is, the solution of the linearized problem reduces to finding
the Green kernel for the $u$-equation alone, which in turn amounts
to solving the resolvent equation for $\cL$ with delta-function
data, or, equivalently, solving the differential equation
\eqref{eq:linLaplace} with source $S=\delta_y(x)$. This we shall do
in standard fashion by writing \eqref{eq:linLaplace} as a
first-order system and solving appropriate jump conditions at $y$
obtained by the requirement that $G_\lambda$ be a distributional
solution of the resolvent equations.

This procedure is greatly complicated by the circumstance that the
resulting
$(n+2)\times(n+2)$
 first-order system
\begin{equation}\label{eq:firsto}
\Theta(x,\lambda)W_x=\A(x,\lambda)W
\end{equation}
is  {\it singular} at the special point where
$A(U)$
vanishes, with
$\Theta$ dropping to rank $n+1$. However, in the end we find as usual
that $G_\lambda$ is uniquely determined by these criteria, not only
for the values $\Re \lambda \ge \eta_0>0$ guaranteed by
$C^0$-semigroup theory/energy estimates, but, as in the usual
nonsingular case \cite{He}, on the {\it set of consistent splitting}
for the first-order system \eqref{eq:firsto}, which includes all of
$\{\Re \lambda \ge  0\}\setminus \{0\}$. This has the implication
that the essential spectrum of $\cL$ is confined to $\{\Re
\lambda<0\}\cup \{0\}$.


\begin{remark}\label{jump}\textup{
The fact (obtained by energy-based resolvent estimates) that
$\cL-\lambda$ is coercive for $\Re \lambda \ge \eta_0$ shows by
elliptic theory that the resolvent is well-defined and unique in
class of distributions for $\Re \lambda $ large, and thus the
resolvent kernel may be determined by the usual construction using
appropriate jump conditions. That is, from standard considerations,
we see that the construction {\it must} work, despite the apparent
wrong dimensions of decaying manifolds (which happens for any $\Re
\lambda >0$).}
\end{remark}

To deal with the singularity of the first-order system is the most
delicate and novel part of the present analysis. It is our hope that
the methods we use here may be of use also in other situations where
the resolvent equation becomes singular, for example in the closely
related situation of relaxation systems discussed in
\cite{MaZ1,MaZ5}.

\section{Construction of the resolvent kernel}\label{sec:resolker}

\subsection{Outline}
In what follows we shall denote $' = \partial_x$ for simplicity; we also write $A(x) = A(U)$ and $B(x) = B(U)$.
Let us now construct the resolvent kernel for  $\cL$, or
equivalently, the solution of \eqref{eq:linLaplace} with
delta-function source in the $u$ component. The novelty in the
present case is the extension of
this standard method
to a situation in which the spectral problem can only be written as
a {\it degenerate} first order ODE. Unlike the real viscosity and
relaxation cases \cite{MaZ1,MaZ3,MaZ4,MaZ5} (where the operator $L$,
although degenerate, yields a non-degenerate first order ODE in an
appropriate reduced space), here we deal with a system of form
\[
\Theta W' = \A(x,\lambda)W,
\]
where
\[
\Theta = \begin{pmatrix} A & \\ & I_2 \end{pmatrix},
\]
is degenerate at $x=0$.

To construct the resolvent kernel we solve
\begin{equation}
\label{eq:restype} (\Theta \partial_x - \A(x,\lambda))
\cG_\lambda(x,y) = \delta_y(x),
\end{equation}
in the distributional sense, so that
\begin{equation}
\label{eq:i} (\Theta \partial_x - \A(x,\lambda)) \cG_\lambda(x,y) =
0,
\end{equation}
in the distributional sense for all $x\ne y$ with appropriate jump
conditions (to be determined) at $x=y$. The first entry of the
three-vector $\cG_\lambda$ is the resolvent kernel $G_\lambda$ of
$\cL$ that we seek.

Namely $\cG_\lambda$, is the solution in the sense of distribution
of system (\ref{eq:linLaplace}) (written in conservation form):
\begin{equation}
    \begin{aligned}
    (Au)' &= - \left ( \lambda  + LB \right )u +L p +
    \delta_{y}(x)\\
    q' &= Bu - p \\
    p' &= - q.
    \end{aligned}
      \label{eq:resolker}
\end{equation}
\subsection{Asymptotic behavior}

First, we study at the asymptotic behavior of solutions to the
spectral system
\begin{equation}
\label{specsyst}
\begin{aligned}
(A(x)u)' &= - (\lambda + LB(x)) u + Lp,\\
q' &= B(x)u - p,\\
p' &= -q,
\end{aligned}
\end{equation}
away from the singularity at $x=0$, and for values of $\lambda \neq
0$, $\Re \lambda \geq 0$. We pay special attention to the small
frequency regime, $\lambda \sim 0$. First, we diagonalize
$A$
 as
\begin{equation}\label{a-diag} \tilde A:=L_p A R_p = \begin{pmatrix}
A_1^- &&0\\&a_p&\\0&&A_2^+\end{pmatrix}
\end{equation}
where
$A_1^-\le-\theta<0,$ $A_2^+\ge \theta>0,$
and $a_p\in \RR$, satisfying
$a_p(+\infty) <0< a_p(-\infty)$. Here, $L_p,R_p$ are bounded
matrices and $L_pR_p = I$. Defining $v:=L_pu$, we rewrite
\eqref{specsyst} as
\begin{equation}
\label{specsys-v}
\begin{aligned}
(\tilde A(x)v)' &= - (\lambda + \tilde L\tilde B+L'_p A R_p)v + \tilde Lp,\\
q' &= \tilde B v - p,\\
p' &= -q,
\end{aligned}
\end{equation}
where
$\tilde L:=L_pL$ and $\tilde B:=BR_p$.
Denote the limits of the coefficient as
\begin{equation}
\tilde A_\pm := \lim_{x\to\pm\infty} \tilde A(x), \qquad \tilde B_\pm := \lim_{x\to\pm\infty} B(x)R_p.
\end{equation}
The asymptotic system thus can be written as
\begin{equation}
\label{asympsyst}
 W' = \A_\pm(\lambda)W,
\end{equation}
where $W = (v,q,p)^\top$, and
\begin{equation}
\A_\pm(\lambda) = \begin{pmatrix} -\tilde A_\pm^{-1}(\lambda +
\tilde L_\pm\tilde B_\pm) & 0 & \tilde A_\pm^{-1}\tilde L \\ \tilde B_\pm & 0 & -1 \\
0 & -1 & 0
\end{pmatrix}.
\end{equation}

To determine the dimensions of the stable/unstable eigenspaces, let
$\lambda \in \R^+$ and $\lambda \to 0,+\infty$, respectively. The
$2\times 2$ lower right-corner matrix clearly gives one strictly
positive and one strictly negative eigenvalues (this later will give
one fast-decaying and one fast-growing modes). In the ``slow''
system (as $|\lambda|\to 0$), eigenvalues are \begin{equation}\label{e-values}\mu_j^\pm(\lambda) = -
\lambda /a_j^\pm + \cO(\lambda^2),\end{equation} where $a_j^\pm$ are eigenvalues of
$A_\pm = A(\pm\infty)$.
Thus, we readily conclude that at $x=+\infty$, there are $p+1$
unstable eigenvalues and $n-p+1$ stable eigenvalues. The stable
$S^+(\lambda)$ and unstable $U^+(\lambda)$ manifolds (solutions
which decay, respectively, grow at $+\infty$) have, thus, dimensions
\begin{equation}
\label{dims+}
\begin{aligned}
\dim U^+(\lambda) &= p+1,\\
\dim S^+(\lambda) &= n-p+1,
\end{aligned}
\end{equation}
in $\Re \lambda > 0$. Likewise, there exist $n-p+1$ unstable
eigenvalues and $p$ stable eigenvalues so that the stable (solutions
which grow at $-\infty$) and unstable (solutions which decay at
$-\infty$) manifolds have dimensions
\begin{equation}
\label{dims-}
\begin{aligned}
\dim U^-(\lambda) &=p,\\
\dim S^-(\lambda) &=n-p+2.
\end{aligned}
\end{equation}
\begin{remark}\textup{
Notice that, unlike customary situations in the Evans function
literature \cite{AGJ,ZH,GZ,MaZ1,MaZ3,San}, here the dimensions of
the stable (resp. unstable) manifolds $S^+$ and $S^-$ (resp. $U^+$
and $U^-$) \emph{do not agree}. Under these considerations, we look
at the dispersion relation
\[
\pi_\pm(i\xi) = -i\xi^3 - A_\pm^{-1}(\lambda + LB_\pm)\xi^2 -i\xi
-A_\pm^{-1} = 0.
\]
For each $\xi \in \R$, the $\lambda$-roots of the last equation define
algebraic curves
\[
\lambda_j^\pm(\xi) \in \sigma (1+LB_\pm \xi)^{-1} (- \xi^2 + i A_\pm
\xi(1+\xi^2)), \quad \xi \in \R,
\]
touching the origin at $\xi =0$. Denote $\Lambda$ as the open
connected subset of $\C$ bounded on the left by the
rightmost envelope of the curves
$\lambda_j^\pm(\xi)$, $\xi \in \R$. Note that the set $\{ \Re \lambda
\geq 0, \lambda \neq 0\}$ is properly contained in $\Lambda$. By
connectedness the dimensions of $U^\pm(\lambda)$ and
$S^\pm(\lambda)$ do not change in $\lambda \in \Lambda$. We define
$\Lambda$ as the set of \emph{(not so) consistent splitting}
\cite{AGJ}, in which the matrices $\A_\pm(\lambda)$ remain
hyperbolic, with not necessarily agreeing dimensions of stable
(resp. unstable) manifolds.}
\end{remark}

\begin{lemma}
\label{lem:asymmodes} For each $\lambda \in
\Lambda$, the spectral system \eqref{asympsyst} associated to the
limiting, constant coefficients asymptotic behavior of
\eqref{specsyst}, has a basis of solutions
\[
e^{\mu^\pm_j(\lambda) x} V_j^\pm(\lambda), \quad x \gtrless 0, \:
j=1,...,n+2.
\]
Moreover, for $|\lambda| \sim 0$, we can find analytic
representations for $\mu_j^\pm$ and $V_j^\pm$, which consist of $2n$
slow modes
\[
\mu_j^\pm(\lambda) = -\lambda/a_j^\pm + \cO(\lambda^2),
\qquad j=2,...,n+1,
\]
and four fast modes,
\[
\begin{aligned}
\mu^\pm_1 (\lambda) &= \pm \theta^\pm_1 + \cO(\lambda), \\
\mu^\pm_{n+2} (\lambda) &= \mp \theta^\pm_{n+2} + \cO(\lambda).
\end{aligned}
\]where $\theta^\pm_1$ and $\theta^\pm_{n+2}$ are positive
constants.
\end{lemma}


In view of the structure of the asymptotic systems, we are able to
conclude that for each initial condition $x_0
> 0$, the solutions to \eqref{specsyst} in $x \geq x_0$ are spanned
by 
decaying/growing modes
\begin{equation}\label{phi+}\begin{aligned}\Phi^+:&=\{\phi^+_1,...,\phi^+_{n-p+1}\},
\\\Psi^+:&=\{\psi_{n-p+2}^+,...,\psi_{n+2}^+\},\end{aligned}\end{equation}
as $x \to +\infty$, whereas for each initial condition $x_0 < 0$,
the solutions to \eqref{specsyst} are spanned in $x < x_0$ by
growing/decaying modes
\begin{equation}\label{phi-}\begin{aligned}\Psi^-:&=\{ \psi_1^-,...,
\psi^-_{n-p+2}\},\\
\Phi^-:&=\{\phi_{n-p+3}^-,...,\phi^-_{n+2}\},\end{aligned}\end{equation}
as $x \to -\infty$.

We rely on the conjugation lemma of \cite{MeZ1} to link such modes
to those of the limiting constant coefficient system
\eqref{asympsyst}.
\begin{lemma}\label{lem-estmodes}
For $|\lambda|$ sufficiently small, there exist growing and decaying
solutions $\psi^\pm_j(x,\lambda),\phi_j^\pm(x,\lambda)$, in $x
\gtrless 0$, of class $C^1$ in $x$ and analytic in $\lambda$,
satisfying
\begin{equation}\label{est-modes}
\begin{aligned}
\psi^\pm_j(x,\lambda) &= e^{\mu_j^\pm(\lambda)x} V_j^\pm(\lambda) (I
+ \cO(e^{-\eta|x|})), \\
\phi^\pm_j(x,\lambda) &= e^{\mu_j^\pm(\lambda)x} V_j^\pm(\lambda) (I
+ \cO(e^{-\eta|x|})),
\end{aligned}
\end{equation}
where $0 < \eta$ is the decay rate of the traveling wave, and
$\mu_j^\pm$ and $V_j^\pm$ are as in Lemma \ref{lem:asymmodes} above.
\end{lemma}
\begin{proof}
This a direct application of the conjugation lemma of \cite{MeZ1}
(see also the related gap lemma in \cite{GZ,ZH,MaZ1,MaZ3}).
\end{proof}

\subsection{Solutions near $x \sim 0$}

Our goal now is to analyze system \eqref{specsyst} close to the
singularity $x=0$. To fix ideas, let us again stick to the case
$x>0$, the case $x<0$ being equivalent. We introduce a ``stretched''
variable $\xi$ as follows:
\begin{equation*}
    \xi = \int_{1}^{x}\frac{dz}{a_p(z)},
\end{equation*}
so that $\xi(1) = 0$, and $\xi \to +\infty$ as $x \to 0^+$. Under
this change of variables we get
\begin{equation*}
    u' = \frac{du}{dx} = \frac{1}{a_p(x)}\frac{du}{d\xi} =
    \frac{1}{a_p(x)}\dot{u},
\end{equation*}
and denoting $\;\dot{ }$ $= d/d\xi$.
In the stretched variables, making some further changes of variables
if necessary, the system \eqref{specsys-v} becomes a
block-diagonalized system at leading order of the form
\begin{equation}
\label{eq:strechtedsyst} \dot{Z} = \begin{pmatrix} -\alpha & 0 \\ 0&
0 \end{pmatrix} + a_p(\xi) \Theta(\xi) Z,
\end{equation} where $\Theta(\xi)$ is some bounded matrix and
$\alpha$ is the $(p,p)$ entry of the matrix
$\lambda + \tilde L\tilde B+L'_p A R_p + \tilde A'$,
noting that $$\alpha(\xi) \ge
\delta_0>0,$$ for some $\delta_0$ and any $\xi$ sufficiently large
or $x$ sufficiently near zero.

The blocks $-\alpha I$ and $0$ are clearly spectrally separated and
the error is of order $\cO(|a_p(\xi)|) \to 0$ as $\xi \to +\infty$.
By the pointwise reduction lemma (see Lemma \ref{pwrl} and Remark
\ref{rem:reduced} below), we can separate the flow into slow and
fast coordinates. Indeed, after proper transformations we separate
the flows on the reduced manifolds of form
\begin{align}
\dot{Z_1} &= - \alpha Z_1 + \cO(a_p) Z_1,\\
\dot{Z_2} &= \cO(a_p) Z_2.
\end{align}

Since $-\alpha \leq -\delta_0 < 0$ for $\lambda \sim 0$ and $\xi
\geq 1/\epsilon$, with $\epsilon > 0$ sufficiently small, and since
$a_p(\xi) \to 0$ as $\xi \to +\infty$, the $Z_1$ mode decay to zero
as $\xi \to +\infty$, in view of
\[
e^{-\int_0^\xi \alpha(z) \, dz} \lesssim e^{-(\Re \lambda + \half
\delta_0) \xi}.
\]

These fast decaying modes correspond to fast decaying to zero
solutions when $x \to 0^+$ in the original $u$-variable. The $Z_2$
modes comprise slow dynamics of the flow as $x \to 0^+$.



\begin{proposition}
\label{prop:smallep} There exists $0 < \epsilon_0 \ll 1$
sufficiently small, such that, in the small frequency regime
$\lambda \sim 0$, the solutions to the spectral system
\eqref{specsyst} in $(-\epsilon_0,0) \cup (0,\epsilon_0)$ are
spanned by fast modes
\begin{equation}
\label{wk} w_{k_p}^\pm(x,\lambda) = \begin{pmatrix} \tilde u_{k_p}^\pm
\\ \tilde q_{k_p}^\pm \\ \tilde p_{k_p}^\pm \end{pmatrix} \qquad \pm\epsilon_0 \gtrless x
\gtrless 0,
\end{equation}
decaying to zero as $x \to 0^\pm$, and slowly varying modes
\begin{equation}
z_j^\pm(x,\lambda) = \begin{pmatrix} \tilde u_j^\pm
\\ \tilde q_j^\pm \\ \tilde p_j^\pm \end{pmatrix},  \qquad \pm\epsilon_0 \gtrless x \gtrless 0,\label{z13}
\end{equation}
with bounded limits as $x \to 0^\pm$.

Moreover, the fast modes \eqref{wk} decay as
\begin{equation}
\label{decayu2} \tilde u_{{k_p}p}^\pm \sim |x|^{\alpha_0} \to 0
\end{equation} and
\begin{equation}
\label{decaypq2}
\begin{pmatrix} \tilde u_{{k_p}j}^\pm \\\tilde q_{k_p}^\pm \\ \tilde p_{k_p}^\pm \end{pmatrix} \sim
\cO(|x|^{\alpha_0}a_p(x)) \to 0, \qquad j\not=p,
\end{equation}
as $x \to 0^\pm$; here, $\alpha_0$ is some positive constant and $u_{k_p}=(u_{{k_p}1},...,u_{{k_p}p},...,u_{{k_p}n})^\top$.
\end{proposition}

\subsection{Two Evans functions}
We first define the following related Evans functions
\begin{equation}
\label{rEvans} D_\pm(y,\lambda) := \det (\Phi^+\;W_{k_p}^{\mp} \;
\Phi^-)(y,\lambda), \qquad \mbox{for } y \gtrless 0,
\end{equation} where $\Phi^\pm$ are defined as in \eqref{phi+}, \eqref{phi-}, and
$W_{k_p}^\pm = (u_{k_p}^\pm,q_{k_p}^\pm,p_{k_p}^\pm)^\top$ are defined as in
\eqref{wk}. Note that ${k_p}$ here is always fixed and equals to
$n-p+2$.

We first observe the following simple properties of $D_\pm$.
\begin{lemma}\label{lem-Evansfns} For $\lambda$ sufficiently small, we have
\begin{equation}\label{Evansfns1}\begin{aligned}
D_\pm(y,\lambda)&=(\det A)^{-1}\gamma_\pm(y)
\Delta\lambda +
\cO(|\lambda|^2),
\end{aligned}\end{equation}
where 
\begin{equation}\label{Lop-det}\begin{aligned}\Delta &:= \det\begin{pmatrix}r_2^+&\cdots&r^+_{k_p-1}&r_{k_p+1}^-&\cdots&r_{n+1}^-&-[u]
\end{pmatrix}
\\
\gamma_\pm(y)&:=\det\begin{pmatrix}q_1^+&q_{k_p}^\mp\\p_1^+&p_{k_p}^\mp\end{pmatrix}_{|_{\lambda=0}}
\end{aligned}\end{equation}
with $[u] = u_+ - u_-$ and $r_j^\pm$ eigenvectors of $(A_\pm)^{-1}(LB)_\pm$, spanning the stable/unstable subspaces at $\pm\infty$, respectively. 
\end{lemma}
\begin{proof}
By our choice, at $\lambda=0$, we can
take\begin{equation}\label{gchoice}\phi_1^+ (x,0)= \phi_{n+2}^-
(x,0) = \bar W_x(x)\end{equation} where $\bar W$ is the shock
profile. By Leibnitz' rule and using \eqref{gchoice}, we compute
$$\begin{aligned} \partial_\lambda D_-(y,0)&=
\det\Big(\partial_\lambda
\phi_1^+,...,\phi_{{k_p}-1}^+,W_{k_p}^+,\phi_{{k_p}+1}^-,...,\phi_{n+2}^-\Big)_{|_{\lambda=0}}+
\cdots\\&\qquad \cdots+\det\Big(
\phi_1^+,...,\phi_{{k_p}-1}^+,W_{k_p}^+,\phi_{{k_p}+1}^-,...,\partial_\lambda\phi_{n+2}^-\Big)_{|_{\lambda=0}},
\end{aligned}$$ where, by using \eqref{gchoice}, only the first and third terms are possibly nonvanishing and thus grouped
together, yielding\begin{equation}\label{der-D}\begin{aligned}
\partial_\lambda D_-(y,0)&=\det\Big(\phi_1^+,...,\phi_{{k_p}-1}^+,
W_{k_p}^+,\phi_{{k_p}+1}^-,...,\phi_{n+1}^-,\partial_\lambda\phi_{n+2}^- -
\partial_\lambda\phi^+_1\Big)_{|_{\lambda=0}}.
\end{aligned}\end{equation}

Recall that $W_{k_p}^+,\phi_j^\pm$ satisfy \begin{equation}\label{eqW}
\Theta W_x = \A(x,\lambda)W,\end{equation} where $W = (u,q,p)$ and
\[
\Theta = \begin{pmatrix} A & \\ & I_2 \end{pmatrix}.
\]
Thus, $\partial_\lambda \phi_1^+(x,\lambda)$ satisfies
$$\Theta (\partial_\lambda \phi_1^+)_x = \A(x,0)\partial_\lambda \phi_1^+(x,0) + \partial_\lambda\A(x,0)\phi_1^+(x,0), $$
which directly gives\begin{equation}\label{eq-W1}(a\partial_\lambda
u_1^+)_x = - L(\partial_\lambda q_1^+)_x - \bar u_x.\end{equation}

Likewise, $\partial_\lambda \phi_{n+2}^-(x,\lambda) =
(\partial_\lambda u_{n+2}^-,\partial_\lambda
q_{n+2}^-,\partial_\lambda p_{n+2}^-)$
satisfies\begin{equation}\label{eq-Wn2}(a\partial_\lambda
u_{n+2}^-)_x = - L(\partial_\lambda q_{n+2}^-)_x - \bar
u_x.\end{equation}

Integrating equations \eqref{eq-W1} and \eqref{eq-Wn2} from
$+\infty$ and $-\infty$, respectively, with use of boundary
conditions $\partial_\lambda \phi_1^+(+\infty) = \partial_\lambda
\phi_{n+2}^-(-\infty) =0$, we obtain
\begin{equation}\begin{aligned}A\partial_\lambda
u_1^+&= - L\partial_\lambda q_1^+ - \bar u + u_+
\\A\partial_\lambda u_{n+2}^-&= - L\partial_\lambda q_{n+2}^- - \bar
u + u_-.\end{aligned}\end{equation}
and thus
\begin{equation}\label{col3}\begin{aligned}A(\partial_\lambda u_{n+2}^--\partial_\lambda
u_1^+)&= - L(\partial_\lambda q_{n+2}^--\partial_\lambda q_1^+) -
[u].\end{aligned}\end{equation}

In addition, we note that $W_{k_p}^+,\phi_j^\pm$ satisfy the equation \eqref{eqW}
and thus
$(Au)' = -Lq'$ with $W_{k_p}^+(+\infty)=\phi_1^+(+\infty)=0$,
$\phi_{n+2}^-(-\infty)=0$, $q_j^\pm(\pm\infty)=0$, and $$\begin{aligned}&u_j^+(+\infty) = (A_+)^{-1}r_j^+, \qquad j=2,...,k_p-1
\\&u_j^-(-\infty) =  (A_-)^{-1}r_j^-, \qquad j=k_p+1,...,n+1.\end{aligned}$$ 
Thus, we integrate the equation $(Au)' = -Lq'$,
yielding
\begin{equation}\label{col1-2}\begin{aligned} A u_j^+ &= -L q_j^+, \qquad \mbox{for  }j=1,{k_p},\\
A u_j^+ &= -L q_j^+ + r_j^+, \qquad \mbox{for  }j=2,...,{k_p-1},\\
A u_j^- &= -L q_j^-+r_j^-, \qquad \mbox{for
}j=k_p+1,...,n+1\\A u_j^- &= -L q_j^-, \qquad \mbox{for  }j=n+2.\end{aligned}\end{equation}
Using estimates \eqref{col1-2} and \eqref{col3}, we can now compute
the $\lambda$-derivative \eqref{der-D} of $D_\pm$ at $\lambda=0$ as
\begin{equation}\label{derD-}\begin{aligned}
\partial_\lambda D_-(y,0)&=\det\begin{pmatrix}u_1^+&\cdots&u_j^+&\cdots&u_{k_p}^+&\cdots&u_j^-&\cdots&\partial_\lambda
u_{n+2}^--\partial_\lambda
u_1^+\\q_1^+&\cdots&q_j^+&\cdots&q_{k_p}^+&\cdots&q_j^-&\cdots&\partial_\lambda
q_{n+2}^--\partial_\lambda
q_1^+\\p_1^+&\cdots&p_j^+&\cdots&p_{k_p}^+&\cdots&p_j^-&\cdots&\partial_\lambda
p_{n+2}^--\partial_\lambda p_1^+\end{pmatrix}\\
&=(\det A)^{-1}\det\begin{pmatrix}0&\cdots&r_j^+&\cdots&0&\cdots&r_j^-&\cdots&-[u]
\\q_1^+&\cdots&q_j^+&\cdots&q_{k_p}^+&\cdots&q_j^-&\cdots&\partial_\lambda
q_{n+2}^--\partial_\lambda
q_1^+\\p_1^+&\cdots&p_j^+&\cdots&p_{k_p}^+&\cdots&p_j^-&\cdots&\partial_\lambda
p_{n+2}^--\partial_\lambda p_1^+\end{pmatrix}
\\
&=(\det A)^{-1}\det\begin{pmatrix}q_1^+&q_{k_p}^+\\p_1^+&p_{k_p}^+\end{pmatrix}
\det\begin{pmatrix}r_2^+&\cdots&r^+_{k_p-1}&r_{k_p+1}^-&\cdots&r_{n+1}^-&-[u]
\end{pmatrix}
\end{aligned}\end{equation}
which proves \eqref{Evansfns1}. 
The proof for $D_+$ follows similarly.
\end{proof}

\begin{lemma}
\label{lemma-mD}
Defining the Evans functions \begin{equation}\label{Evansfns-def}
D_\pm(\lambda): = D_\pm(\pm 1, \lambda),\end{equation} we then have
\begin{equation}\label{E-consistent}
D_+(\lambda) = m D_-(\lambda) +\cO(|\lambda|^2)\end{equation} where
$m$ is some nonzero factor.
\end{lemma}
%

\begin{proof}
Proposition \ref{prop:smallep} gives
\begin{equation}
w_{k_p}^\pm(x) = \begin{pmatrix} \tilde u_{k_p}^\pm
\\ \tilde q_{k_p}^\pm \\ \tilde p_{k_p}^\pm \end{pmatrix} = \cO(|x|^{\alpha_0}),
\end{equation}as $x\to 0$, where $\alpha_0$ is defined as in Proposition \ref{prop:smallep}, which
 guarantees an existence of positive constants $\epsilon_1,\epsilon_2$
near zero such that
\begin{equation*}
w_{k_p}^+(-\epsilon_1) =
w_{k_p}^-(+\epsilon_2).
\end{equation*}
Thus, this together with the fact that $w_{k_p}^\pm$ are solutions of the ODE \eqref{eqW} yields
\begin{equation*}w_{k_p}^+(-1) = m_{k_p}
w_{k_p}^-(+1)
\end{equation*}
for some nonzero constant $m_{k_p}$.
Putting these estimates into \eqref{Evansfns1} and using analyticity
of $D_\pm$ in $\lambda$ near zero, we easily obtain the conclusion.
\end{proof}

\section{Resolvent kernel bounds in low--frequency regions} In this section, we shall
derive pointwise bounds on the resolvent kernel $G_{\lambda}(x,y)$
in low-frequency regimes, that is, $|\lambda| \to 0$. For
definiteness, throughout this section, we consider only the case
$y<0$. The case $y>0$ is completely analogous by symmetry.

We solve \eqref{eq:resolker} with the jump conditions at $x=y$:
\begin{equation}
\label{eq:conduy} [\cG_\lambda(.,y)] = \begin{pmatrix} A(y)^{-1}
&0&0\\0& 1&0\\0&0&1\end{pmatrix}
\end{equation}
where,
working on diagonalized coordinates (see \eqref{specsys-v}),
we can assume that
$A$
is of diagonal form as in \eqref{a-diag},
\[
A = \begin{pmatrix}
A_1^- &&0\\&a_p&\\0&&A_2^+\end{pmatrix},
\]
with $A_1^- \leq - \theta < 0, A_2^+(y)\ge \theta> 0$.
Meanwhile, we can write $\cG_\lambda(x,y)$ in  terms of decaying
solutions at $\pm \infty$ as follows
\begin{equation}
\label{eq:formcolu} \cG_\lambda(x,y) = \begin{cases}
\Phi^+(x,\lambda)C^+(y,\lambda) + W_{k_p}^+(x,\lambda)C_{k_p}^+(y,\lambda), & x>y,\\
- \Phi^-(x,\lambda)C^-(y,\lambda), & x<y.\end{cases}
\end{equation}
where $C_j^\pm$ are row vectors.
We compute
the coefficients $C_j^\pm$ by means of the transmission conditions
\eqref{eq:conduy} at $y$. Therefore, solving by Cramer's rule the
system
\begin{equation}\label{eq:systC}
\begin{pmatrix}
\Phi^+ & W_{k_p}^+ & \Phi^-
\end{pmatrix}
\begin{pmatrix}C^+ \\ C_{k_p}^+ \\ C^-
\end{pmatrix}_{\displaystyle{|(y,\lambda)}} =
\begin{pmatrix} A(y)^{-1}
&0&0\\0& 1&0\\0&0&1\end{pmatrix},
\end{equation}
we readily obtain,
\begin{align}\label{C-form}
\begin{pmatrix}C^+ \\ C_{k_p}^+ \\ C^-
\end{pmatrix}(y,\lambda) &= D_-(y,\lambda)^{-1}\begin{pmatrix} \Phi^+
& W_{k_p}^+ & \Phi^-
\end{pmatrix}^{adj}\begin{pmatrix}A(y)^{-1}
&0&0\\0& 1&0\\0&0&1\end{pmatrix}
\end{align}
where $M^{adj}$ denotes the adjugate matrix of a matrix
$M$. Note that
\begin{align}
C_{jp}^\pm(y,\lambda) &= a_p(y)^{-1}D_-(y,\lambda)^{-1}\begin{pmatrix} \Phi^+
& W_{k_p}^+ & \Phi^-
\end{pmatrix}^{pj}(y,\lambda), \label{Cjp}
\\
C_{jl}^\pm(y,\lambda) &= \sum_k D_-(y,\lambda)^{-1}\begin{pmatrix} \Phi^+
& W_{k_p}^+ & \Phi^-
\end{pmatrix}^{kj}(y,\lambda)(A(y)^{-1})_{kl},\qquad l\not = p,\label{Cjl}
\end{align}
where $()^{ij}$ is the determinant of the $(i,j)$ minor, and
$(A(y)^{-1})_{kl}$,
$l\not=p$, are bounded in $y$.

We then easily obtain the following.
\begin{lemma}\label{lem-estCnear0} For $y$ near zero, we have
\begin{equation}\label{est-C13}\begin{aligned}
C_1^+(y,\lambda) &=~~~\frac {1}{\lambda}v_0([u])+ \cO(1),\\
C_{n+2}^-(y,\lambda) &= -\frac {1}{\lambda}v_0([u]) +
\cO(1),
\end{aligned}\end{equation} where
$v_0([u])$ is some constant vector depending only on $[u]$
and \begin{equation}\label{est-C2}\begin{aligned} C_{k_p}^+(y,\lambda)
&=a_p(y)^{-1}|y|^{-\alpha_0}\cO(1),\\
C_j^+(y,\lambda)&=\cO(1)\qquad 1<j<k_p,\\
C_j^-(y,\lambda)&=\cO(1)\quad k_p<j<n+2,
\end{aligned}\end{equation}
where $k_p=n-p+2$, $\alpha_0$ is defined as in Proposition \ref{prop:smallep} and $\cO(1)$ is a uniformly
bounded function, probably depending on $y$ and $\lambda$.
\end{lemma}
\begin{proof} 
We shall first estimate $C_{n+2,p}^-(y,\lambda)$. Observe that 
$$\begin{pmatrix} \Phi^+
& W_{k_p}^+ & \Phi^-
\end{pmatrix}^{p,n+2}(y,\lambda) = \begin{pmatrix} \Phi^+
& W_{k_p}^+ & \Phi^-
\end{pmatrix}^{p,n+2}(y,0) + \cO(\lambda)$$
where by the same way as done in Lemma 2.5 we obtain an estimate
$$\begin{pmatrix} \Phi^+
& W_{k_p}^+ & \Phi^-
\end{pmatrix}^{p,n+2}(y,0) =a_p(\det A)^{-1}\gamma_-(y)
\Delta^{p,n+2},
$$ where $\gamma_-(y)$ and $\Delta$ are defined as in \eqref{Lop-det}, and $\Delta^{p,n+2}$ denotes the minor determinant. 
Thus, recalling 
\eqref{Evansfns1}
and \eqref{Cjp}, we can estimate
$C_{n+2,p}^-(y,\lambda)$ as
\begin{align*}
C_{n+2,p}^-(y,\lambda) &= a_p(y)^{-1}D_-(y,\lambda)^{-1}\begin{pmatrix} \Phi^+
& W_{k_p}^+ & \Phi^-
\end{pmatrix}^{p,n+2}(y,\lambda)\\
&= - \frac1{\lambda}\Delta^{-1}\Delta^{p,n+2} + \cO(1),
\end{align*} where $\cO(1)$ is uniformly bounded since
$a_p(y)D_-(y,\lambda)$ and normal modes $\phi^\pm_j$ are all bounded
uniformly in $y$ near zero. Similar computations can be done for $C_{n+2,l}^-(y,\lambda)$. Thus,
we obtain the bound for $C_{n+2}^-$ as
claimed.
 The bound for $C_{1}^+$ follows similarly, noting that
$\phi_{n+2}^- \equiv \phi_1^+$ at $\lambda=0$.

For the estimate on $C_{k_p}^+$, we first observe that by view of
\eqref{Evansfns1}, with noting that $\det(A)\sim a_p(y)$ as $|y|\to0$, and the estimate \eqref{wk} on $w_{{k_p}p}^+$,
\begin{equation}\label{D-est}|D_-(y,\lambda)| \ge \theta|\lambda| |y|^{\alpha_0},\end{equation} for some
$\theta>0$. This together with the fact that $\phi_{n+2}^- \equiv \phi_1^+$ at
$\lambda=0$ yields the estimate for $C_{k_p}^+$ as claimed.

We next estimate $C_j^+$ (resp. $C_j^-$) for $1<j<k_p$ (resp. $k_p<j<n+2$). We note that by view of estimate \eqref{wk} on $W_{k_p}$,
$$\begin{pmatrix} \Phi^+
& W_{k_p}^+ & \Phi^-
\end{pmatrix}^{pj} = \cO(\lambda)\cO( |y|^{\alpha_0}a_p(y))$$ and for $k\not=p$,
$$\begin{pmatrix} \Phi^+
& W_{k_p}^+ & \Phi^-
\end{pmatrix}^{kj} = \cO(\lambda)\cO( |y|^{\alpha_0})$$
These estimates together with \eqref{D-est} and \eqref{Cjl},\eqref{Cjp} immediately yield estimates for $C_j^\pm$ as claimed.
\end{proof}

\begin{proposition}[Resolvent kernel bounds as $|y|\to 0$]\label{prop-nearzero} 
For
$y$ near zero, there hold
\begin{equation}\label{G0-est1} G_\lambda(x,y) ={\lambda^{-1}}\bar W_x v_0([u]) + \cO(1)\sum_{a_j^+>0}e^{-(\lambda/a_j^+ + \cO(\lambda^2))x} + \cO(e^{-\theta |x|})
\end{equation} for $y<0<x$, and
\begin{equation}\label{G0-est2} G_\lambda(x,y) = {\lambda^{-1}}\bar W_x v_0([u])+
\cO(1) \Big(1+\frac{|x|^{\alpha}}{a_p(y)|y|^\alpha}\Big)
\end{equation}
for $y<x<0$, and
\begin{equation}\label{G0-est3}G_\lambda(x,y) ={\lambda^{-1}}\bar W_x v_0([u])+ \cO(1)\sum_{a_j^-<0}e^{-(\lambda/a_j^- + \cO(\lambda^2))x} + \cO(e^{-\theta |x|})
\end{equation} for $x<y<0$.

Similar bounds can be obtained for the case $y>0$.
\end{proposition}
\begin{proof} For the case $y<0<x$, using \eqref{est-C13} and recalling that
$\phi_1^+(x) = \bar W_x + \cO(\lambda)e^{-\theta |x|}$ and $W_{k_p}^+(x)
\equiv 0$, we have
$$\begin{aligned}G_\lambda(x,y) &=  \Phi^+(x) C^+(y) = \sum_{j=1}^{k_p-1}\phi_j^+(x)C^+_j(y)\\&= \Big(\bar W_x +
\cO(\lambda)e^{-\theta |x|}\Big)\Big(\frac
{1}{\lambda}v_0([u])+ \cO(1)\Big) + \cO(1)\sum_{j=2}^{k_p-1}e^{\mu_j^+ x},\end{aligned}$$ yielding
\eqref{G0-est1}; here, we recall that $$\mu_j^\pm = -\lambda/a_{j}^\pm + \cO(\lambda^2)$$ with $a_j^+>0$ for $j=2,...,k_p-1$ and $a_j^-<0$ for $j=k_p+1,...,n+1$ ($a_j^\pm$ are necessarily eigenvalues of
$A_\pm$
). In the second case $y<x<0$, from the formula
\eqref{eq:formcolu}, 
we have
$$G_\lambda(x,y) = \Phi^+(x,\lambda) C^+(y,\lambda)+
W_{k_p}^+(x,\lambda)C_{k_p}^+(y,\lambda)$$ where the first term contributes
${\lambda^{-1}}v_0([u])\bar W_x + \cO(1)$ as in the first case, and
the second term is estimated by \eqref{est-C2} and \eqref{decayu2}.

Finally, we estimate the last case $x<y<0$ in a same way as done in
the first case, noting that $y$ is still near zero and $W_{n+2}^-(x) =
\bar W_x + \cO(\lambda)e^{-\theta |x|}$. \end{proof}

Next, we estimate the kernel $G_\lambda(x,y)$ for $y$ away from zero.
Note however that the representations \eqref{eq:formcolu} and above
estimates fail to be useful in the $y \to -\infty$ limit, since we
actually need precise decay rates in order to get an estimate of
form
\[
|G_\lambda(x,y)| \leq C e^{-\eta|x-y|},
\]
which are unavailable from $\phi_j^+$ in the $y \to -\infty$ regime.
Thus, we need to express the $(+)$-bases in terms of the growing
modes $\psi_j^-$ at $-\infty$, and the decaying mode $\phi_j^-$
where $\psi_j^-,\phi_j^-$ are defined as in Lemma
\ref{lem-estmodes}. Expressing such solutions in the basis for $y <
0$, away from zero, there exist {\it analytic} coefficients
$d_{jk}(\lambda), e_{jk}(\lambda)$ such that
\begin{equation}
\label{eq:goodmodes} \begin{aligned} \phi_j^+(x,\lambda)&=
\sum d_{jk}(\lambda) \phi_k^-(x,\lambda)+\sum e_{jk}(\lambda) \psi_k^-(x,\lambda)\\
W_{k_p}^+(x,\lambda)&=
\sum d_{k_pk}(\lambda) \phi_k^-(x,\lambda)+\sum e_{k_pk}(\lambda) \psi_k^-(x,\lambda) .
\end{aligned}
\end{equation}

Furthermore, for our convenience, we define the following adjoint
normal modes
\begin{equation}\label{adjoint}\begin{pmatrix} \tilde \Psi^- & \tilde \Phi^-\end{pmatrix}
:= \begin{pmatrix} \Psi^- & \Phi^-\end{pmatrix}^{-1} \Theta^{-1}.
\end{equation}
We then obtain the following estimates.
\begin{lemma}\label{lem-estadj} For $|\lambda|$ sufficiently small and $|x|$ sufficiently large,
\begin{equation}\label{est-adjmodes}
\begin{aligned}
\tilde \psi^-_j(x,\lambda) & =\cO(e^{-\mu_j^-(\lambda)x})\tilde V_j^-(\lambda)(I
+ \cO(e^{-\theta|x|})), \\
\tilde \phi^-_j(x,\lambda) &= \cO(e^{-\mu_j^-(\lambda)x})\tilde V_j^-(\lambda)(I +
\cO(e^{-\theta|x|}))
\end{aligned}
\end{equation}where $\mu_j^-$ are defined as in Lemma
\ref{lem-estmodes}.
\end{lemma}
\begin{proof} The proof is clear from the estimates of
$\psi^-_j,\phi_j^-$ in \eqref{est-modes}.
\end{proof}

\begin{lemma}\label{lem-Cjk} We have
\begin{align}\label{Cjk+}
C_j^+(y,\lambda) &= \sum c^+_{jk}(\lambda) \tilde \psi_k^-(y,\lambda)^*\\
C_j^-(y,\lambda) &= \sum c^-_{jk}(\lambda) \tilde
\psi_k^-(y,\lambda)^* + 
\tilde \phi_j^-(y,\lambda)^*,
\label{Cjk-}
\end{align}for meromorphic coefficients $c^\pm_{jk}$ in $\lambda$.
\end{lemma}
\begin{proof} The proof follows by using \eqref{eq:goodmodes}, definition \eqref{adjoint}, and property of computing determinants.
\end{proof}

We then have the following representation for $G_\lambda(x,y)$, for $y$ large.
\begin{prop}\label{prop-greenlow}
Under the assumptions of Theorem \ref{theo-main},
 for $|\lambda|$ sufficiently small and $|y|$ sufficiently large, we have
\begin{equation} \label{greenlow}
G_\lambda(x,y)=\sum_{j,k}c_{jk}^+(\lambda)\phi_j^+(x,\lambda)\tilde{\psi}_k^-(y,\lambda)^*,
\end{equation}
for $y<0<x$, and
\begin{equation} \label{greenlow+}G_\lambda(x,y)=\sum_{j,k}d^+_{jk}(\lambda)\phi_j^-(x,\lambda)\tilde{\psi}_k^-(y,\lambda)^*-
\sum_{k}\psi_k^-(x,\lambda)\tilde{\psi}_k^-(y,\lambda)^*,
\end{equation}
for $y<x<0$,  and
\begin{equation} \label{greenlow-}G_\lambda(x,y)=\sum_{j,k}d^-_{jk}(\lambda)\phi_j^-(x,\lambda)\tilde{\psi}_k^-(y,\lambda)^*+
\sum_{k}\phi_k^-(x,\lambda)\tilde{\phi}_k^-(y,\lambda)^*,
\end{equation}
for $x<y<0$, where $c_{jk}^+(\lambda),d_{jk}^\pm(\lambda)$ are
scalar meromorphic functions satisfying
$$c^+ = \begin{pmatrix}-I_{k_p}&0\end{pmatrix}\begin{pmatrix}\Phi^+ & W_{k_p}^+ & \Phi^-\end{pmatrix}^{-1}\Psi^-$$ and $$d^\pm = \begin{pmatrix}0&-I_{n-k_p}\end{pmatrix}\begin{pmatrix}\Phi^+ & W_{k_p}^+ & \Phi^-\end{pmatrix}^{-1}\Psi^-.$$
\end{prop}

\begin{proof} Using representation \eqref{eq:formcolu} of $G_\lambda(x,y)$ together with \eqref{Cjk+} and \eqref{Cjk-}, we easily obtain the expansions \eqref{greenlow} and \eqref{greenlow-}, respectively. For \eqref{greenlow+}, again, using \eqref{Cjk+}, \eqref{eq:goodmodes}, and \eqref{eq:formcolu}, we can write
\begin{equation} \label{G-exp1}\begin{aligned}G_\lambda(x,y)&=\sum_{j,k}d^+_{jk}(\lambda)\phi_j^-(x,\lambda)\tilde{\psi}_k^-(y,\lambda)^*+
\sum_{j,k}e_{jk}^+\psi_j^-(x,\lambda)\tilde{\psi}_k^-(y,\lambda)^*\\&=\begin{pmatrix}\Psi^-& \Phi^-\end{pmatrix}(x) \begin{pmatrix}e^+\\d^+\end{pmatrix}
\tilde\Psi^-(y)^*
\end{aligned}
\end{equation}
Meanwhile, by \eqref{eq:formcolu} and \eqref{C-form},
\begin{equation} \label{G-exp2}G_\lambda(x,y)= \begin{pmatrix}\Phi^+ & W_{k_p}^+ &0\end{pmatrix}(x) \begin{pmatrix}\Phi^+ & W_{k_p}^+ & \Phi^-\end{pmatrix}^{-1}(y) \Theta ^{-1}(y)
\end{equation}
In view of the definition \eqref{adjoint} of $\tilde \Psi^-,\tilde \Phi^-$, \eqref{G-exp1} and \eqref{G-exp2} yield
\begin{equation*}\begin{aligned}\begin{pmatrix}e^+\\d^+\end{pmatrix}
&=\begin{pmatrix}\tilde\Psi^-& \tilde\Phi^-\end{pmatrix}\Theta\begin{pmatrix}\Phi^+ & W_{k_p}^+ &0\end{pmatrix}
\begin{pmatrix}\Phi^+ & W_{k_p}^+ & \Phi^-\end{pmatrix}^{-1}\Psi^-
\\&=\begin{pmatrix}\Psi^-& \Phi^-\end{pmatrix}^{-1}\Big[I - \begin{pmatrix}0&\Phi^-\end{pmatrix}
\begin{pmatrix}\Phi^+ & W_{k_p}^+ & \Phi^-\end{pmatrix}^{-1}\Big]\Psi^-
\\&=\begin{pmatrix}I_{k_p}\\0 \end{pmatrix} -\begin{pmatrix}0&0\\0&I_{n-k_p}\end{pmatrix}\begin{pmatrix}\Phi^+ & W_{k_p}^+ & \Phi^-\end{pmatrix}^{-1}\Psi^-,
\end{aligned}
\end{equation*}
which proves the proposition. \end{proof}

\begin{proposition}[Resolvent kernel bounds as $|y|\to +\infty$]\label{prop-awayzero}
Make the assumptions of Theorem \ref{theo-main}.
Then, for
$|y|$ large, defining $$E_\lambda (x,y)= \lambda^{-1}\sum_{a_j^->0}\tilde V_{j,0}^-e^{(\lambda/a_j^-+\cO(\lambda^2))y} \bar W_x(x),$$
there hold
\begin{equation}\label{G1-est1} \begin{aligned}G_\lambda&(x,y) =E_\lambda(x,y)\\&+ \cO(1)\Big(\sum_{a_j^->0} e^{(\lambda/a_j^-+\cO(\lambda^2))y}+\cO(e^{-\theta|y|})\Big)
\Big(\sum_{a_k^->0}e^{(-\lambda/a_k^-+\cO(\lambda^2))x}+\cO(e^{-\theta|x|})\Big)\end{aligned}
\end{equation} for $y<0<x$, and
\begin{equation}\label{G1-est2} \begin{aligned}G_\lambda(x,y) =&E_\lambda(x,y) +\cO(1)\sum_{a_j^->0} e^{(-\lambda/a_j^-+\cO(\lambda^2))(x-y)}\\&+\cO(1)\sum_{a_j^->0,\; a_k^-<0} e^{(\lambda/a_j^-+\cO(\lambda^2))y}e^{(-\lambda/a_k^-+\cO(\lambda^2))x}
+\cO(e^{-\theta(|x-y|)})\end{aligned}
\end{equation} for $y<x<0$, and
\begin{equation}\label{G1-est3}\begin{aligned}G_\lambda(x,y) =&E_\lambda(x,y) +\cO(1)\sum_{a_j^-<0} e^{(-\lambda/a_j^-+\cO(\lambda^2))(x-y)}\\&+\cO(1)\sum_{a_j^->0,\; a_k^-<0} e^{(\lambda/a_j^-+\cO(\lambda^2))y}e^{(-\lambda/a_k^-+\cO(\lambda^2))x}
+\cO(e^{-\theta(|x-y|)})\end{aligned}
\end{equation} for $x<y<0$.

Similar
bounds can be obtained for the case $y>0$.
\end{proposition}

\begin{proof} The proof follows directly from the representations of $G_\lambda(x,y)$ derived in Proposition \ref{prop-greenlow} and the corresponding estimates on normal modes, noting that
$$|c^+_{jk}|,|d_{jk}^\pm| = \left\{\begin{matrix} \cO(\lambda^{-1}) & j=1,\\\cO(1)&\mbox{otherwise.} \end{matrix}\right.$$
Indeed, we recall, for instance, that
$$c^+_{jk} = D^{-1}_-\begin{pmatrix}-I_{k_p}&0\end{pmatrix}\begin{pmatrix}\Phi^+ & W_{k_p}^+ & \Phi^-\end{pmatrix}^{kj}\Psi^-,$$where $()^{kj}$ denotes the determinant of the $(k,j)$ minors. For the case $j\not=1$, by using the fact that we choose $\phi_1^{+} \equiv \phi_{n+2}^- \equiv \bar W_x$ at $\lambda=0$, determinant of the $(k,j)$ minor therefore has the order one in $\lambda$, which cancels out the $\lambda^{-1}$ term coming from our spectral stability condition: $|D_-^{-1}| \le \cO(\lambda^{-1})$.

\end{proof}


\section{Pointwise 
bounds 
and low-frequency estimates} \label{sec:ptwise}


In this section, using the previous pointwise bounds (Propositions
\ref{prop-nearzero} and \ref{prop-awayzero}) for the resolvent
kernel in low-frequency regions, we derive pointwise bounds for the
``low-frequency'' Green function:
\begin{equation}\label{LFGreenfn} G^I(x,t;y): = \frac{1}{2\pi
i}\int_{\Gamma\bigcap\{|\lambda|\le r\}}e^{\lambda
t}G_\lambda(x,y)d\lambda\end{equation}
where $\Gamma$ is any contour near zero, but away from the essential spectrum.

\begin{proposition} \label{prop-greenbounds} 
Under the assumptions of Theorem \ref{theo-main},
defining the effective
diffusion
$\beta_\pm:=(L_pL B R_p)_\pm$
(see \eqref{a-diag}), the low-frequency
Green distribution $G^I(x,t;y)$ associated with the linearized
evolution equations  may be decomposed as
\begin{equation}
G^I(x,t;y)= E+  \widetilde G^I + R, \label{ourdecomp}
\end{equation}
where, for $y<0$:
\begin{equation}\label{E}
E(x,t;y):= \sum_{a_k^->0}\bar U_x(x)\tilde V_{k,0}^-e_k(y,t),
\end{equation}
\begin{equation}\label{e}
  e_k(y,t):=\left(\textrm{errfn }\left(\frac{y+a_{k}^-t}{\sqrt{4\beta_{-}t}}\right)
  -\textrm{errfn }\left(\frac{y-a_k^{-}t}{\sqrt{4\beta_{-}t}}\right)\right);
\end{equation}
\begin{equation}\label{GIbounds}
\begin{aligned}
|\partial_{x}^\gamma \partial_y^\beta  \widetilde G^I(x,t;y)|\le  C&t^{-(|\alpha|+|\gamma|)/2}\Big( \sum_{k=1}^n
t^{-1/2}e^{-(x-y-a_k^{-} t)^2/Mt} \\
&+\sum_{a_k^{-} < 0, \, a_j^{-} > 0} \chi_{\{ |a_k^{-} t|\ge |y| \}}
t^{-1/2} e^{-(x-a_j^{-}(t-|y/a_k^{-}|))^2/Mt}
 \Big), \\
\end{aligned}
\end{equation}
\begin{equation}
\begin{aligned}
R(x,t;y)&= \cO(e^{-\eta(|x-y|+t)}) + \cO(e^{-\eta
t})\chi(x,y)\Big[1+\frac{1}{a_p(y)}(x/y)^\alpha\Big],
\end{aligned}
\label{Rbounds}
\end{equation}for some $\eta$, $C$, $M>0$, where $0\le |\beta|, |\gamma| \le 1$, $\alpha = \frac{LB(0)+a_p'(0)}{|a_p'(0)|}$ and
$$\chi(x,y) =\left\{\begin{matrix}1,&&-1<y<x<0\\0,&&
\mbox{otherwise.}\end{matrix}\right.$$

Symmetric bounds hold for $y\ge 0$.
\end{proposition}

\begin{proof} Having the resolvent kernel estimates in Propositions \ref{prop-nearzero} and
\ref{prop-awayzero}, we can now follow the previous analyses of
\cite{ZH,MaZ1,MaZ3}. Indeed, the claimed bound for $E$ precisely
comes from the $\lambda^{-1}$ term.
Likewise, estimates of $\widetilde G^I$ are due to bounds in
Proposition \ref{prop-awayzero} for $y$ away from zero and those in
Proposition \ref{prop-nearzero} for $y$ near zero but $x$ away from
zero. The singularity occurs only in the case $-1<y<x<0$, as
reported in Proposition \ref{prop-nearzero}. In this case, using the
estimate \eqref{G0-est2} and moving the contour $\Gamma$ in
\eqref{LFGreenfn} into the stable half-plane $\{\Re\lambda<0\}$, we
have
$$\int_\Gamma e^{\lambda t}\Big(1+\frac{|x|^{\alpha}}{a_p(y)|y|^\alpha}\Big)d\lambda
= \cO(e^{-\eta t})\Big(1+\frac{|x|^{\alpha}}{a_p(y)|y|^\alpha}\Big),$$
which precisely contributes to the second term in $R(x,t;y)$. The
first term in $R(x,t;y)$ is as usual the fast decaying term.
\end{proof}


With the above pointwise estimates on the (low-frequency) Green
function, we have the following from \cite{MaZ1,MaZ3}.

\begin{lemma}[\cite{MaZ1,MaZ3}]\label{lem-estGI} 
Under the assumptions of Theorem \ref{theo-main},
$\widetilde G^I$ satisfies
\begin{equation}\label{estGI}
\Big|\int_{-\infty}^{+\infty} \partial_y^\beta\widetilde
G^I(\cdot,t;y) f(y)dy \Big|_{L^p} \le C (1+t)^{-\frac 12
(1/q-1/p)-|\beta|/2}|f|_{L^q},
\end{equation}
for all $t\ge 0$, some $C>0$, for any $1\le q\le p$.
\end{lemma}

We recall the following fact from \cite{Z4}.
\begin{lemma}[\cite{Z4}]\label{lem-kernel-e} The kernel $e$
satisfies
\begin{equation} \label{e-bound}\begin{aligned}
&|e_y(\cdot, t)|_{L^p}, |e_t(\cdot, t)|_{L^p}, \le C t^{-\frac 12
(1-1/p)},\\
&|e_{yt}(\cdot, t)|_{L^p}\le C t^{-\frac 12 (1-1/p)-1/2}.
\end{aligned}
\end{equation}
for all $t> 0$, some $C>0$, for any $p\ge 1$.
\end{lemma}



Finally, we have the following estimate on $R$ term.
\begin{lemma}\label{lem-estR}
Under the assumptions of Theorem \ref{theo-main},
$R(x,t;y)$ satisfies
\begin{equation} \label{S1-bound}
\Big|\int_{-\infty}^{+\infty}R(\cdot,t;y) f(y)dy \Big|_{L^p} \le C
e^{-\eta t}(|f|_{L^p} + |f|_{L^\infty}),
\end{equation}
for all $t\ge 0$, some $C,\eta>0$, for any $1\le p\le \infty$.
\end{lemma}

\begin{proof} The estimate clearly holds for the fast decaying term $e^{-\eta(|x-y|+t)}$ in
$R$. Whereas, to estimate the second term, first notice that it is
only nonzero precisely when $-1<y<x<0$ or $0<x<y<1$. Thus, for
instance, when $-1<x<0$, we estimate
$$\begin{aligned}
\Big|\int_{-\infty}^{+\infty}\chi(x,y)&\Big[1+\frac{1}{a_p(y)}(x/y)^\alpha\Big]f(y)dy
\Big|=\Big|\int_{-1}^{x}\Big[1+\frac{1}{a_p(y)}(x/y)^\alpha\Big]f(y)dy
\Big|
\\&\le C|f|_{L^\infty}\Big[1+\int_{-1}^{x}\frac{1}{|a_p(y)|}(x/y)^\alpha dy\Big]
\\&\le C|f|_{L^\infty},
\end{aligned}$$
where the last integral is bounded by that fact that $a_p(x) \sim x$
as $|x|\to 0$. From this, we easily obtain
$$\begin{aligned}
\Big|\int_{-\infty}^{+\infty}e^{-\eta
t}&\chi(x,y)\Big[1+\frac{1}{a_p(y)}(x/y)^\alpha\Big]f(y)dy
\Big|_{L^p(-1,0)}\le Ce^{-\eta t}|f|_{L^\infty},
\end{aligned}$$
which proves the lemma.\end{proof}

\begin{remark}\label{rem-singterm}\textup{We note here that the singular term
$a_p^{-1}(y)(x/y)^{\alpha}$ appearing in \eqref{G0-est2} and
\eqref{Rbounds} contributes in the time-exponential decaying term.
This thus agrees with the resolvent kernel for the scalar
convected-damped equation $u_t + a_pu_x=-LBu,$ for which we can find
explicitly the Green function as a convected time-exponential decaying
delta function similar as in the relaxation or real viscosity case.}
\end{remark}

\section{Nonlinear damping estimate and high--frequency estimate}\label{sec:damping}

In this section, we 
establish an auxiliary damping energy estimate. We first recall the
nonlinear perturbation equations with $(u,q)$ perturbation variables
 \begin{equation}
     \begin{aligned}
     u_{t}+ (A(u)u)_x +Lq_{x} &=\dot \alpha (U_x+u_x),\\
 -q_{xx} + q +(B(u)u)_{x} &=0,
    \end{aligned}
    \label{eqpert}
\end{equation}
where
 we now denote
\begin{equation}\label{AB}
A(u):= Df(U+u),\,\qquad  B(u):= Dg(U+u).\end{equation}

We prove the following:
\begin{proposition}\label{prop-damping}
Under the assumptions of Theorem \ref{theo-main},
so long as $\|u\|_{W^{2,\infty}}$ and $|\dot \alpha|$
remain smaller than a small constant $\zeta$ and the amplitude $|U_x|$ is sufficiently small,
there holds
\begin{equation}\label{damp-est}
\|u\|_{H^k}^2(t)\le Ce^{-\theta t}\|u\|_{H^{k}}^2(0)+ C\int_0^t
e^{-\theta(t-s)}(\|u\|_{L^2}^2+ |\dot\alpha|^2)(s)ds, \qquad
\theta>0,
\end{equation}for $k=1,...,4$.
\end{proposition}

\begin{proof} Let us work for the case $\dot\alpha \equiv 0$. The
general case will be seen as a straightforward extension. We first observe that \begin{equation}\label{coeff-est}|A_{0x}|,|A_{0t}|,|A_x|,|A_t|,|B_x|,|B_t| = \cO(|U_x|+\zeta)\end{equation} where $A,B$ are defined as in \eqref{AB} and $A_0$ the symmetrizer matrix as in \eqref{S1}.

We note that from the second equation of
\eqref{eqpert} we easily obtain
\begin{equation}\label{estq}\|q\|_{H^k}\le
C\|u\|_{H^{k-1}},\end{equation} for $k\ge1$. Meanwhile, from the first equation, we estimate
\begin{equation*}
\begin{aligned}\frac12\dt\iprod{A_0u,u} &= \iprod{A_0u_t,u} + \frac12\iprod{A_{0t}u,u} \\&= -\iprod{A_0A_xu+A_0Au_x + Lq_x,u}+ \frac12\iprod{A_{0t}u,u}\\& = -\iprod{A_0A_xu - \frac12(A_0A)_xu + Lq_x,u}+ \frac12\iprod{A_{0t}u,u},
\end{aligned}
\end{equation*}
which, by \eqref{coeff-est} and \eqref{estq}, yields
\begin{equation}\label{keyineq0th}
\begin{aligned}\frac12\dt\iprod{A_0u,u} &\le C\|u\|_{L^2}^2.
\end{aligned}
\end{equation}

Now, to obtain the estimates \eqref{damp-est} in the case of $k=1$, we compute
\begin{equation}\label{1-est}\begin{aligned}
\frac 12 \dt\iprod{A_0u_x,u_x} &=\iprod{(A_0u_t)_x,
u_x} + \frac12 \iprod{A_{0t}u_x,u_x} - \iprod{A_{0x}u_t,u_x}\\&=-\iprod{(A_0Au_x + A_0Lq_x)_x,
u_x} + \frac12 \iprod{A_{0t}u_x,u_x} - \iprod{A_{0x}u_t,u_x}\\&=-\iprod{A_0Lq_{xx},u_x} - \iprod{A_0Au_{xx},u_x},
 + \iprod{\cO(|U_x|+\zeta)u_x,u_x}  + \|q\|_{H^1}^2\\&=-\iprod{A_0Lq_{xx},u_x} + \iprod{\cO(|U_x|+\zeta)u_x,u_x}  + \cO(1)\|u\|_{L^2}^2\\&=-\iprod{A_0LBu_x,u_x} + \iprod{\cO(|U_x|+\zeta)u_x,u_x}  + \cO(1)\|u\|_{L^2}^2,
\end{aligned}\end{equation}
noting that since $A_0A$ is symmetric, we have $$-\iprod{A_0Au_{xx},u_x} = \frac 12\iprod{(A_0A)_xu_x,u_x}  =  \iprod{\cO(|U_x|+\zeta)u_x,u_x}.$$

Likewise, in spirit of Kawashima-type estimates, we compute
\begin{equation}\label{K-est}\begin{aligned}
\frac12\dt\iprod{Ku,u_x} &=\frac12\iprod{K_tu,u_x} + \frac12\iprod{Ku_t,u_x} +\frac12\iprod{Ku,u_{xt}}\\&=\frac12\iprod{K_tu,u_x} + \frac12\iprod{Ku_t,u_x} -\frac12\iprod{Ku_x,u_{t}}-\frac12\iprod{K_xu,u_t}\\&=\iprod{Ku_t,u_x} +\frac12\iprod{K_tu,u_x}   -\frac12\iprod{K_xu,u_t} \\&=-\iprod{KAu_x+KA_xu+KLq_x,u_x} +\frac12\iprod{K_tu,u_x}   -\frac12\iprod{K_xu,u_t}
\\&=-\iprod{KAu_x,u_x}+ \iprod{\cO(|U_x|+\zeta)u_x,u_x}  + \cO(1)\|u\|_{L^2}^2.
\end{aligned}\end{equation}

Adding \eqref{1-est} and \eqref{K-est} together, we obtain
\begin{equation}\begin{aligned}
\frac12\dt\Big(\iprod{Ku,u_x}&+\iprod{A_0u_x,u_x}\Big) \\&= -\iprod{(KA+A_0LB)u_x,u_x}+ \iprod{\cO(|U_x|+\zeta)u_x,u_x}  + \cO(1)\|u\|_{L^2}^2
\end{aligned}\end{equation}
which, by the Kawashima-type condition \eqref{KALB}: $KA+A_0LB\ge \theta$ and the fact that $\cO(|U_x|+\zeta)$ is sufficiently small, yields
\begin{equation}\label{keyineq1st}\begin{aligned}
\frac12\dt\Big(\iprod{Ku,u_x}+\iprod{A_0u_x,u_x}\Big)\le  -\frac12\theta\iprod{u_x,u_x}+ \cO(1)\|u\|_{L^2}^2
\end{aligned}\end{equation}

Similarly, for $k\ge1$, paying attention to the leading terms, we can compute
\begin{equation*}\begin{aligned}
\frac 12 \dt\iprod{A_0\Dxk u,\Dxk u} &=\iprod{A_0\Dxk u_t,\Dxk u}+\frac 12 \iprod{A_{0t}\Dxk u,\Dxk u}\\&=\iprod{\Dxk(A_0u_t),
\Dxk u} + \iprod{\cO(|U_x|+\zeta)\Dxk u,\Dxk u} + \cO(1)\|u\|_{H^{k-1}}^2,\end{aligned}\end{equation*}
where by using the first equation and then the second one, we obtain
\begin{equation*}\begin{aligned}\iprod{\Dxk(A_0u_t),
\Dxk u}&=-\iprod{\Dxk(A_0Au_x + A_0A_xu + A_0Lq_x),\Dxk u}\\&= - \iprod{A_0L\partial_x^{k-1}q_{xx},
\Dxk u}-\iprod{A_0A\partial_x^{k+1}u,\Dxk u} + \cdots\\&= - \iprod{A_0L\partial_x^{k}(Bu),
\Dxk u}+\frac12\iprod{(A_0A)_x\Dxk u,\Dxk u} + \cdots.\end{aligned}\end{equation*}
Thus, we have obtained
\begin{equation}\label{kth-est}\begin{aligned}
\frac 12 \dt&\iprod{A_0\Dxk u,\Dxk u} \\& = - \iprod{A_0LB \Dxk u,
\Dxk u}+ \iprod{\cO(|U_x|+\zeta)\Dxk u,\Dxk u} + \cO(1)\|u\|_{H^{k-1}}^2.\end{aligned}\end{equation}

Meanwhile, we have the following $k^{th}$-order Kawashima-type energy estimate
\begin{equation}\label{K-kthest}\begin{aligned}
\frac12\dt\iprod{K\partial_x^{k-1}u,\Dxk u} &=\iprod{K\partial_x^{k-1}u_t,\Dxk u} +\frac12\iprod{K_t\partial_x^{k-1}u,\Dxk u}   -\frac12\iprod{K_x\partial_x^{k-1} u,\partial_x^{k-1}u_t}\\&=-\iprod{KA\Dxk u,\Dxk u}+ \iprod{\cO(|U_x|+\zeta)\Dxk u,\Dxk u}  + \cO(1)\|u\|_{H^{k-1}}^2.
\end{aligned}\end{equation}

Hence, as before, adding \eqref{kth-est} and \eqref{K-kthest} together and using the Kawashima-type condition \eqref{KALB}: $KA+A_0LB\ge \theta$ and the fact that $\cO(|U_x|+\zeta)$ is sufficiently small, we obtain
\begin{equation}\label{keyineqkth}\begin{aligned}
\frac12\dt\Big(\iprod{K\partial_x^{k-1}u,\Dxk u}+\iprod{A_0\Dxk u,\Dxk u}\Big)\le  -\frac12\theta\iprod{\Dxk u,\Dxk u} + \cO(1)\|u\|_{H^{k-1}}^2.
\end{aligned}\end{equation}

Now, for $\delta>0$, let us define $$\cE(t):= \sum_{k=0}^s\delta^k\Big(\iprod{K\partial_x^{k-1}u,\Dxk u}+\iprod{A_0\Dxk u,\Dxk u}\Big).$$
By applying the standard Cauchy's inequality on $\iprod{K\partial_x^{k-1}u,\Dxk u}$ and using the positive definiteness of $A_0$, we observe that $\cE(t) \sim \|u\|_{H^k}^2$. We then use the above estimates \eqref{keyineq0th},\eqref{keyineq1st}, \eqref{keyineqkth}, and take $\delta$ sufficiently small to derive
\begin{equation}\label{key-est} \dt\cE(t) \le -\theta_3 \cE(t) + C
\|u\|_{L^2}^2(t)\end{equation}for some $\theta_3>0$, from which
\eqref{damp-est} follows by the standard Gronwall's inequality.
\end{proof}

With the damping nonlinear energy estimates in hands, we immediately obtain the following estimates for high-frequency part of the solution operator $e^{\cL t}$:
\begin{equation}\label{formS2}\begin{aligned}\cS_2(t)
&=\frac{1}{2\pi i}\int_{-\theta_1-i\infty}^{-\theta_1+i\infty}
\chi_{\{|\Imag\lambda|\geq\theta_2\}}e^{\lambda t}
(\lambda-\cL)^{-1}
d\lambda,
\end{aligned}\end{equation}for small positive numbers $\theta_1,\theta_2$; see \eqref{iLT}. Here, $\chi_{\{|\Imag\lambda|\geq\theta_2\}}$ equals to $1$ for $|\Imag\lambda|\geq\theta_2$ and zero otherwise.

\begin{proposition}[High-frequency estimate] \label{prop-HFest} 
Under the assumptions of Theorem \ref{theo-main},
\begin{equation}\label{HFsoln-est}
\begin{aligned}\|\cS_2(t)f\|_{L^2} &\le C
e^{-\theta_1 t}\|f\|_{H^{2}},\\\|
\partial^\alpha_x\cS_2(t)f\|_{L^{2}}&\le  C
e^{-\theta_1t}\|f\|_{H^{\alpha+2}},\end{aligned}\end{equation}for
some $\theta_1>0$.
\end{proposition}


Proof of the proposition follows exactly in a same way as done in our companion paper \cite{LMNPZ1} for the scalar case. We recall it here for sake of completeness. The first step is to estimate the solution
of the resolvent system
\begin{equation*}
\begin{aligned}
    \lambda u + (A\,u)_{x} +L q_{x} &=\varphi,\\
    - q_{xx} +q +(B\, u)_{x} &=\psi,
\end{aligned}
\end{equation*}
where
$A(x)=Df(U)$ and $B(x)=Dg(U)$
 as before.
\medskip

\begin{proposition}[High-frequency bounds]\label{prop-resHF}
Under the assumptions of Theorem \ref{theo-main},
for some $R,C$ sufficiently large
and $\gamma>0$ sufficiently small, we obtain
\begin{equation*}
\begin{aligned}
	|(\lambda - \cL)^{-1}(\varphi-L\partial_x (\cK \psi))|_{H^1}
	&\le C \Big( |\varphi|_{H^1}^2+|\psi|_{L^2}^2 \Big),\\
	|(\lambda - \cL)^{-1}(\varphi-L\partial_x (\cK \psi))|_{L^2}
	&\le \frac{C}{|\lambda|^{1/2}}\Big(|\varphi|_{H^1}^2+|\psi|_{L^2}^2\Big),
\end{aligned}
\end{equation*}
for all $|\lambda|\ge R$ and $\Real\lambda \ge -\gamma$, where $\cK:=(-\partial_x^2+ 1)^{-1}$.
\end{proposition}
\medskip

\begin{proof}
A Laplace transformed version of the nonlinear energy estimates
\eqref{damp-est} in Section \ref{sec:damping} with $k = 1$
(see \cite{Z7}, pp. 272--273, proof of Proposition 4.7 for further details) yields
\begin{equation}\label{Re-est}
	\begin{aligned}
		\Big(\Real\lambda+\frac{\gamma_1}{2}\Big)|u|_{H^1}^2\le C\Big(|u|_{L^2}^2
		+ |\varphi|_{H^1}^2+|\psi|_{L^2}^2\Big).
	\end{aligned}
\end{equation}
On the other hand, taking the imaginary part of the $L^2$ inner product of $U$
against $\lambda u = \cL u + \partial_xL\cK h + f$ and applying the Young's inequality,
we also obtain the standard estimate
\begin{equation}\label{Im-est}
	\begin{aligned}
		|\Imag\lambda||u|_{L^2}^2&\le |\iprod{\cL u,u}|
		+ |\iprod{L\cK \psi,u_x}| + |\iprod{\varphi,u}| \\
		&\le C \Big(|u|_{H^1}^2 + |\psi|_{L^2}^2 + |\varphi|_{L^2}^2\Big),
	\end{aligned}
\end{equation}
noting the fact that $\cL$ is a bounded operator from $H^1$ to $L^2$
and $\cK$ is bounded from $L^2$ to $H^1$.

Therefore, taking $\gamma=\gamma_1/4$, we obtain from \eqref{Re-est} and \eqref{Im-est}
\begin{equation*}
	|\lambda||u|_{L^2}^2 + |u|_{H^1}^2\le C \Big(|u|_{L^2}^2
	+ |\psi|_{L^2}^2 + |\varphi|_{H^1}^2\Big),
\end{equation*}
for any $\Real\lambda \ge -\gamma$.
Now take $R$ sufficiently large such that $|u|_{L^2}^2$ on the right  hand side
of the above can be absorbed into the left hand side for $|\lambda|\ge R$, thus yielding
\begin{equation*}
	|\lambda||u|_{L^2}^2 + |u|_{H^1}^2
	\le C \Big( |\psi|_{L^2}^2 + |\varphi|_{H^1}^2\Big),
\end{equation*}
for some large $C>0$, which gives the result as claimed.
\end{proof}

Next, we have the following
\medskip

\begin{proposition}[Mid-frequency bounds]\label{prop-resMF}
Under the assumptions of Theorem \ref{theo-main},
\begin{equation*}
	|(\lambda - \cL)^{-1}\varphi|_{L^2} \le C\,|\varphi|_{H^1}
	\quad \textrm{ for } \;
	R^{-1}\le |\lambda|\le R \mbox{ and }\Real\lambda \ge -\gamma,
\end{equation*}
for any $R$ and $C=C(R)$ sufficiently large and $\gamma = \gamma(R)>0$
sufficiently small.
\end{proposition}
\medskip

\begin{proof}
Immediate, by compactness of the set of frequency under consideration together
with the fact that the resolvent $(\lambda-\cL)^{-1}$ is analytic with respect to $H^{1}$
in $\lambda$; see, for instance, \cite{Z4}.
\end{proof}
\medskip

With Propositions \ref{prop-resHF} and \ref{prop-resMF} in hand, we are now ready to give:
\medskip

\begin{proof}[Proof of Proposition \ref{prop-HFest}]
The proof starts with the following resolvent identity, using analyticity on the
resolvent set $\rho(\cL)$ of the resolvent $(\lambda-\cL)^{-1}$, for all  $\varphi\in \mathcal{D}(\cL)$,
\begin{equation*}
(\lambda-\cL)^{-1}\varphi=\lambda^{-1}(\lambda-\cL)^{-1}\cL \varphi+\lambda^{-1}\varphi.
\end{equation*}
Using this identity and \eqref{formS2}, we estimate
\begin{equation*}
	\begin{aligned}
		\cS_2(t)\varphi &=\frac{1}{2\pi i}\int_{-\gamma_1-i\infty}^{-\gamma_1+i\infty}
		\chi_{{}_{\{|\Imag\lambda|\geq\gamma_2\}}}e^{\lambda t}
		\lambda^{-1}(\lambda-\cL)^{-1}\cL\,\varphi\,d\lambda\\
		&\quad+\frac{1}{2\pi i}\int_{-\gamma_1-i\infty}^{-\gamma_1+i\infty}
		\chi_{{}_{\{|\Imag\lambda|\geq\gamma_2\}}} e^{\lambda t}\lambda^{-1}\varphi\,d\lambda\\
		&=:S_1 + S_2,
	\end{aligned}
\end{equation*}
where, by Propositions \ref{prop-HFest} and \ref{prop-resMF}, we have
\begin{equation*}
	\begin{aligned}
		|S_1|_{L^2}&\le C \int_{-\gamma_1-i\infty}^{-\gamma_1+i\infty}
		|\lambda|^{-1}e^{\Real \lambda t}|(\lambda-\cL)^{-1}\cL
		\varphi|_{L^2}|d\lambda|\\
		&\le C e^{-\gamma_1 t}\int_{-\gamma_1-i\infty}^{-\gamma_1+i\infty} |\lambda|^{-3/2}|\cL
		\varphi|_{H^1}|d\lambda|\\
		&\le C e^{-\gamma_1t}|\varphi|_{H^{2}}
	\end{aligned}
\end{equation*}
and
\begin{equation*}
	\begin{aligned}
	|S_2|_{L^2}&\leq\frac{1}{2\pi }\Big|\varphi\int_{-\gamma_1-i\infty}^{-\gamma_1+i\infty}
	\lambda^{-1}e^{\lambda t} d\lambda\Big|_{L^2}+ \frac{1}{2\pi }
	\Big|\varphi\int_{-\gamma_1-i r}^{-\gamma_1+i r}\lambda^{-1}e^{\lambda t} d\lambda \Big|_{L^2}\\
	&\leq Ce^{-\gamma_1 t} |\varphi|_{L^2},
\end{aligned}
\end{equation*}
by direct computations, noting that the integral in $\lambda$ in the first term is identically zero.
This completes the proof of the bound for the term involving $\varphi$ as stated in the proposition.
The estimate involving $\psi$ follows by observing that $L\,\partial_x \cK$ is bounded from $H^s$ to $H^s$.
Derivative bounds can be obtained similarly.
\end{proof}
\medskip

\begin{remark}\label{rem-HF}\rm
We note that in our treating the high-frequency terms by energy estimates (as also done
in \cite{KZ,NZ2,LMNPZ1}), we are ignoring the pointwise contribution there, which would also be
convected time-decaying delta functions.
To see these features, a simple exercise is to do the Fourier transform of the equations
about a constant state.
\end{remark}


\section{Nonlinear analysis}\label{sec:nonlinear}
In this section, we shall prove the main nonlinear stability
theorem. The proof follows exactly word by word as in the scalar case \cite{LMNPZ1}. We present its sketch here for sake of completeness. Define the nonlinear perturbation
\begin{equation}\label{per-variable}\begin{pmatrix}u\\q \end{pmatrix}(x,t): =
\begin{pmatrix}\tilde u\\\tilde q \end{pmatrix}(x+\alpha(t),t) -
\begin{pmatrix}U\\Q
\end{pmatrix}(x),
\end{equation}
where the shock location $\alpha(t)$ is to be determined later.

Plugging \eqref{per-variable} into \eqref{eq:systemold}, we obtain
the perturbation equation
\begin{equation}\label{perteqs}
    \begin{aligned}
    u_{t}+ (Au)_x  +L q_{x} &= N_1(u)_x + \dot\alpha(t)(u_x+U_x),\\
     - q_{xx} +  q +(Bu)_{x} &= N_2(u)_x,
    \end{aligned}
\end{equation}
where $N_j(u)=O(|u|^2)$ so long as $u$ stays uniformly bounded.

We recall the Green function decomposition
\begin{equation}\label{Greendecomp}G(x,t;y) = G^I(x,t;y) +G^{II}(x,t;y)\end{equation} where
$G^I(x,t;y)$ is the low-frequency part. We further define as in
Proposition \ref{prop-greenbounds},
$$\widetilde G^I(x,t;y) = G^I(x,t;y) - E(x,t;y)  - R(x,t;y)$$
and $$\widetilde G^{II}(x,t;y) = G^{II}(x,t;y) + R(x,t;y).$$

Then, we immediately obtain the following from Lemmas
\ref{lem-estGI}, \ref{lem-estR} and Proposition \ref{prop-HFest}:
\begin{lemma}\label{lem-estGI+II} We obtain
\begin{equation} \label{est-tGI}
\Big|\int_{-\infty}^{+\infty} \partial_y^\beta\widetilde
G^I(\cdot,t;y) f(y)dy \Big|_{L^p} \le C (1+t)^{-\frac 12
(1/q-1/p)-|\beta|/2}|f|_{L^q},
\end{equation}
for all $1\le q\le p, \beta=0,1,$ and
\begin{equation}\label{est-tGII}
\Big|\int_{-\infty}^{+\infty}\widetilde G^{II}(x,t;y)f(y)dy
\Big|_{L^p} \le C e^{-\eta t}|f|_{H^3},
\end{equation} for all $2\le p\le \infty$.
\end{lemma}
\begin{proof} \eqref{est-tGI} is precisely the estimate
\eqref{estGI} in Lemma \ref{lem-estGI}, recalled here for our
convenience. \eqref{est-tGII} is a straightforward combination of
Lemma \ref{lem-estR} and Proposition \ref{prop-HFest}, followed by a
use of the interpolation inequality between $L^2$ and $L^\infty$ and
an application of the standard Sobolev imbedding. \end{proof}

We next show that by Duhamel's principle we have:
\begin{lemma} We obtain the {reduced integral
representation:}
\begin{equation}\label{int-rep}
\begin{aligned}
u(x,t)=& \int_{-\infty}^{+\infty} (\widetilde G^I+\widetilde G^{II})(x,t;y)u_0(y)dy \\
&- \int_0^t \int_{-\infty}^{+\infty} \widetilde G_y^I(x,t-s;y)
\Big(\partial_y L \cK N_2(u)  + N_1(u) + \dot\alpha(t)u\Big)(y,s)\, dy\, ds\\
&+ \int_0^t \int_{-\infty}^{+\infty}\widetilde G^{II}(x,t-s;y)
\Big(\partial_y L \cK N_2(u)  + N_1(u) +
\dot\alpha(t)u\Big)_y(y,s)\, dy\, ds,
\\q(x,t) =& \, (\cK \partial_x)( N_2(u)-Bu) (x,t),
\end{aligned}
\end{equation}
and
\begin{equation}\label{alpha-rep}
\begin{aligned}
\alpha(t)=& -\int_{-\infty}^{+\infty}e_t(y,t)u_0(y)dy \\
&+ \int_0^t \int_{-\infty}^{+\infty}e_{y}(y,t-s)\Big(\partial_y L
\cK N_2(u) + N_1(u) + \dot\alpha(t)u\Big)(y,s)\, dy\, ds.
\end{aligned}
\end{equation}
\begin{equation}\label{alphader-rep}
\begin{aligned}
\dot\alpha(t)=& -\int_{-\infty}^{+\infty}e_t(y,t)u_0(y)dy \\
&+ \int_0^t \int_{-\infty}^{+\infty}e_{yt}(y,t-s)\Big(\partial_y L
\cK N_2(u) + N_1(u) + \dot\alpha(t)u\Big)(y,s)\, dy\, ds.
\end{aligned}
\end{equation}
\end{lemma}

\begin{proof} By Duhamel's principle and the fact that
$$\int_{-\infty}^{+\infty}G(x,t;y) U_x(y)dy = e^{\cL t} U_x(x) =
U_x(x),$$ we obtain
\begin{equation}
\begin{aligned}
u(x,t)=& \int_{-\infty}^{+\infty} G(x,t;y)u_0(y)dy \\
&+ \int_0^t \int_{-\infty}^{+\infty}G(x,t-s;y)
\Big(\partial_y L \cK N_2(u)  + N_1(u) + \dot\alpha(t)u\Big)_y(y,s)\, dy\, ds\\
&+\alpha(t)U_x.
\end{aligned}
\end{equation}
Thus, by defining the {\it instantaneous shock location:}
\begin{align*}
\alpha(t)=& -\int_{-\infty}^{+\infty}e_t(y,t)u_0(y)dy \\
&+ \int_0^t \int_{-\infty}^{+\infty}e_{y}(y,t-s)\Big(\partial_y L
\cK N_2(u) + N_1(u) + \dot\alpha(t)u\Big)(y,s)\, dy\, ds
\end{align*}
and using the Green function decomposition \eqref{Greendecomp}, we
easily obtain the integral representation as claimed in the lemma.
\end{proof}

With these preparations, we are now ready to prove the main theorem,
following the standard stability analysis of \cite{MaZ4,Z3,Z4}:
\begin{proof}[Proof of Theorem \ref{theo-main}]
Define \begin{equation}\label{zeta}
\begin{aligned}\zeta(t):=\sup_{0\le s\le t, 2\le p\le \infty}
\Big[|u(s)|_{L^p}&(1+s)^{\frac 12(1-1/p)}
+|\alpha(s)|+|\dot\alpha(s)|(1+s)^{1/2}\Big].\end{aligned}
\end{equation}

 We shall prove here that for all $t\ge
0$ for which a solution exists with $\zeta(t)$ uniformly bounded by
some fixed, sufficiently small constant, there holds
\begin{equation}\label{zeta-est}
\zeta(t) \le C(|u_0|_{L^1\cap H^s}+\zeta(t)^2) .\end{equation}

This bound together with continuity of $\zeta(t)$ implies that
\begin{equation}\label{zeta-est1} \zeta(t) \le 2C|u_0|_{L^1\cap H^s}\end{equation}
for $t\ge0$, provided that $|u_0|_{L^1\cap H^s}< 1/4C^2$. This would
complete the proof of the bounds as claimed in the theorem, and thus
give the main theorem.

By standard short-time theory/local well-posedness in $H^s$, and the
standard principle of continuation, there exists a solution $u\in
H^s$ on the open time-interval for which $|u|_{H^s}$ remains
bounded, and on this interval $\zeta(t)$ is well-defined and
continuous. Now, let $[0,T)$ be the maximal interval on which
$|u|_{H^s}$ remains strictly bounded by some fixed, sufficiently
small constant $\delta>0$. By Proposition \ref{prop-damping}, and
the Sobolev embeding inequality $|u|_{W^{2,\infty}}\le C|u|_{H^s}$,
$s\ge3$, we have
\begin{equation}\label{Hs}\begin{aligned}|u(t)|_{H^s}^2 &\le Ce^{-\theta t}|u_0|_{H^s}^2
+ C \int_0^t e^{-\theta(t-\tau)}\Big(|u(\tau)|_{L^2}^2+
|\dot\alpha|^2 \Big)d\tau\\&\le C(|u_0|_{H^s}^2
+\zeta(t)^2)(1+t)^{-1/2}.
\end{aligned}\end{equation}
and so the solution continues so long as $\zeta$ remains small, with
bound \eqref{zeta-est1}, yielding existence and the claimed bounds.

Thus, it remains to prove the claim \eqref{zeta-est}. First by
representation \eqref{int-rep} for $u$, for any $2\le p\le \infty$,
we obtain
\begin{equation}
\begin{aligned}
|u|_{L^p}(t)\le& \Big|\int_{-\infty}^{+\infty} (\widetilde G^I+\widetilde G^{II})(x,t;y)u_0(y)dy \Big|_{L^p}\\
&+ \int_0^t\Big| \int_{-\infty}^{+\infty} \widetilde G_y^I(x,t-s;y)
\Big(\partial_y L \cK N_2(u)  + N_1(u) + \dot\alpha(s)u\Big)(y,s)\, dy\Big|_{L^p} ds\\
&+ \int_0^t \Big|\int_{-\infty}^{+\infty}\widetilde G^{II}(x,t-s;y)
\Big(\partial_y L \cK N_2(u)  + N_1(u) +
\dot\alpha(t)u\Big)_y(y,s)\, dy\Big|_{L^p} ds\\
=&I_1 + I_2 + I_3,
\end{aligned}
\end{equation}
where estimates \eqref{est-tGI} and \eqref{est-tGII} yield
\begin{equation*}
\begin{aligned}
I_1&= \Big|\int_{-\infty}^{+\infty} (\widetilde G^I+\widetilde
G^{II})(x,t;y)u_0(y)dy \Big|_{L^p} \\&\le C(1+t)^{-\frac
12(1-1/p)}|u_0|_{L^1} + Ce^{-\eta t}|u_0|_{H^3} \\&\le
C(1+t)^{-\frac 12(1-1/p)}|u_0|_{L^1\cap H^3},
\end{aligned}
\end{equation*}
and, with noting that $\partial_y L \cK$ is bounded from $L^2$ to
$L^2$,
\begin{align*} I_2 &= \int_0^t\Big| \int_{-\infty}^{+\infty}
\widetilde G_y^I(x,t-s;y) \Big(\partial_y L \cK N_2(u)  + N_1(u) +
\dot\alpha(t)u\Big)(y,s)\, dy\Big|_{L^p} ds
\\&\le C \int_0^t(t-s)^{-\frac 12 (1/2-1/p)-1/2}(|u|_{L^\infty} + |\dot\alpha|)|u|_{L^2}(s)ds
\\&\le C \zeta(t)^2\int_0^t(t-s)^{-\frac 12 (1/2-1/p)-1/2}(1+s)^{-3/4}ds
\\&\le C \zeta(t)^2(1+t)^{-\frac 12 (1-1/p)},
\end{align*}
and, together with \eqref{Hs}, $s\ge 4$,
\begin{equation*}
\begin{aligned}
I_3&= \int_0^t \Big|\int_{-\infty}^{+\infty}\widetilde
G^{II}(x,t-s;y) \Big(\partial_y L \cK N_2(u)  + N_1(u) +
\dot\alpha(s)u\Big)_y(y,s)\, dy\Big|_{L^p} ds
\\&\le C\int_0^t e^{-\eta (t-s)} |\partial_y L \cK N_2(u)  + N_1(u) +
\dot\alpha(t)u|_{H^4}(s)ds\\&\le C\int_0^t e^{-\eta (t-s)}
(|u|_{H^s} + |\dot\alpha|)|u|_{H^s}(s)ds\\&\le C(|u_0|_{H^s}^2
+\zeta(t)^2)\int_0^t e^{-\eta (t-s)} (1+s)^{-1}ds
\\&\le C(|u_0|_{H^s}^2
+\zeta(t)^2)(1+t)^{-1}.
\end{aligned}
\end{equation*}

Thus, we have proved
\begin{equation}\label{finalest-u}|u(t)|_{L^p}(1+t)^{\frac
12(1-1/p)}\le C(|u_0|_{L^1\cap H^s} +\zeta(t)^2).\end{equation}

Similarly, using representations \eqref{alpha-rep} and
\eqref{alphader-rep} and the estimates in Lemma \ref{lem-kernel-e}
on the kernel $e(y,t)$, we can estimate (see, e.g., \cite{MaZ4,Z4}),
\begin{equation}\label{finalest-alpha}|\dot\alpha(t)|(1+t)^{1/2} +
|\alpha(t)|\le C(|u_0|_{L^1} +\zeta(t)^2).\end{equation}

This completes the proof of the claim \eqref{zeta-est}, and thus the
result for $u$ as claimed. To prove the result for $q$, we observe
that $\cK\partial_x$ is bounded from $L^p\to W^{1,p}$ for all $1\le
p\le \infty$, and thus from the representation \eqref{int-rep} for
$q$, we estimate
\begin{equation}\label{finalest-q}\begin{aligned} |q|_{W^{1,p}}(t)
&\le C(|N_2(u)|_{L^p}+ |u|_{L^p})(t) \\&\le C|u|_{L^p}(t)\le
C|u_0|_{L^1\cap
H^s}(1+t)^{-\frac12(1-1/p)}\end{aligned}\end{equation} and
\begin{equation}\label{finalest-qHs}\begin{aligned} |q|_{H^{s+1}}(t)
&\le C|u|_{H^s}(t)\le C|u_0|_{L^1\cap
H^s}(1+t)^{-1/4},\end{aligned}\end{equation} which complete the
proof of the main theorem. \end{proof}


\appendix


\section{Spectral stability in the small-amplitude regime}

In this section we verify the spectral stability condition for small-amplitude profiles. Denoting $A = A(U(x))$, $B = B(U(x))$ we have the associated linearized spectral problem
\begin{equation}
\label{evalsyst}
\begin{aligned}
\lambda u + (Au)_x + Lq_x &=0,\\
-q_{xx} + q + (Bu)_x &= 0.
\end{aligned}
\end{equation}

Using the zero-mass conditions
\[
\int u \, dx = 0, \qquad \int q \, dx = 0,
\]
we recast system \eqref{evalsyst} in terms of the integrated coordinates, which we denote, again, as $u$ and $q$. The resulting system reads
\begin{eqnarray}
\lambda u + Au_x + Lq_x &=& 0, \label{k1}\\
-q_{xx} + q + Bu_x &=& 0. \label{k2}
\end{eqnarray}

In what follows we assume that the shocks are weak, that is, $u_\pm \in \mathcal{N}(u_*)$, being $\mathcal{N}$ a neighborhood of a certain state $u_*$, for which
\[
0 < \max_{u \in \mathcal{N}} |u-u_*| \leq \epsilon \ll 1,
\]
with $\epsilon > 0$ sufficiently small; clearly,
\[
|u_* - u_\pm|, |u_- - u_+| = \cO(\epsilon)
\]
and the shock profile for $U$ is approximately scalar, satisfying,
\begin{equation}
\begin{aligned}
U_x &= \cO(\epsilon^2) e^{-\eta \epsilon |x|}(r_p(u_*) + \cO(\epsilon)),\\
U_{xx} &=  \cO(\epsilon^3) e^{-\theta \epsilon |x|},
\end{aligned}
\end{equation}
for some $\theta,\eta > 0$. For the principal characteristic field $a_p := a_p(U(x))$ we have
\begin{align}
(a_p)_x &= \cO(U_x) < 0, \quad \text{(monotonicity)},\label{monotonicity}\\
(a_p)_{xx} &= \cO(U_{xx}). \nonumber
\end{align}

We shall make use of the following
\begin{lemma}
Under \eqref{S0} - \eqref{S2}, there exists a scalar function $\beta = \beta(u) > 0$, such that
\begin{equation}
\label{k17}
(A_0 L)^\top = \beta B,
\end{equation}
for all $u \in \cU$.
\end{lemma}
\begin{proof}
Follows by elementary linear algebra facts, since $A_0LB$ is positive semi-definite with rank one and can be written as $z \otimes w$, for some vectors $z$ and $w$. It follows the existence of a scalar $\beta$, such that $z = \beta w$; it is clearly nonzero and positive because of positive semi-definiteness of $A_0LB$.
\end{proof}
We start by providing some basic Friedrichs-type energy estimates.

\begin{lemma}
Assume $u,q$ and $\Re \lambda \geq 0$ solve \eqref{k1} - \eqref{k2}. If $\epsilon > 0$ is sufficiently small, then there hold the estimates
\begin{equation}
\label{friedRe}
(\Re \lambda) |u|^2_{L^2} + |q|^2_{L^2} + |q_x|_{L^2}^2 \leq C\int |U_x||u|^2 \, dx
\end{equation}
\begin{equation}
\label{friedIm}
|\Im \lambda| \int |U_x||u|^2 \, dx \leq C \int |U_x| \big(\delta |u|^2 + \delta^{-1} |q|^2 \big) \, dx
\end{equation}
for some $C > 0$ and any $\delta > 0$.
\end{lemma}
\begin{proof}
Multiply \eqref{k1} by $A_0 := A_0(U(x))$ and take the complex $L^2$ product against $u$; taking its real part and denoting
\[
\bA := (A_0A((U(x)), \qquad \bL := A_0(U(x))L,
\]
we obtain
\[
(\Re \lambda)\langle u,A_0u \rangle + \Re \langle u, \bA u_x \rangle + \Re \langle u, \bL q_x \rangle = 0.
\]
Using symmetry of $\bA$ and integrating by parts we get
\begin{equation}
\label{k16}
(\Re \lambda)\langle u,A_0u \rangle - \half \Re \langle u, \bA_x u \rangle + \Re \langle u, \bL q_x \rangle = 0.
\end{equation}

Multiply \eqref{k2} by $\beta := \beta(U(x))$, use \eqref{k17}, take the $L^2$ product against $q$, integrate by parts and take its real part. This yields
\begin{equation}
\label{k20}
c^{-1} |q_x|^2_{L^2} + c^{-1} |q|^2_{L^2} + \Re \langle q, \beta_x q \rangle - \Re \langle u, \bL q_x \rangle - \Re \langle \bL_x q, u \rangle = 0,
\end{equation}
because $\beta \geq c^{-1} > 0$. Since the error terms can be absorbed
\[
\beta_x, \bL_x = \cO(|U_x|) = \cO(\epsilon^2),
\]
for $\epsilon$ sufficiently small, and since $A_0$ is positive definite, we obtain inequality \eqref{friedRe}. Inequality \eqref{friedIm} follows in a similar fashion, with the parameter $\delta$ arising after application of Young's inequality.
\end{proof}

\begin{corollary}
There hold the estimates
\begin{equation}
0 \leq \Re \lambda \leq C\epsilon^2, \label{boundRe}
\end{equation}
\begin{equation}
|\Im \lambda| \leq C \epsilon, \label{boundIm}
\end{equation}
for some $C > 0$.
\end{corollary}
\begin{proof}
Estimate \eqref{boundRe} follows immediately from \eqref{friedRe}. Taking $\delta = \epsilon > 0$ in \eqref{friedIm}, and using \eqref{friedRe} to control $|q|^2_{L^2}$ we can easily obtain
\[
(|\Im \lambda| - C\epsilon)\int |U_x||u|^2 \leq 0,
\]
yielding \eqref{boundIm}.
\end{proof}

\subsection{Kawashima-type estimate}

Next we carry out an energy estimate for $u_x$ of Kawashima-type (see \cite{HuZ1,MaZ6}).

\begin{lemma}
For each $\Re \lambda \geq 0$, $\lambda \neq 0$, there holds
\begin{equation}
\label{k39}
|u_x|_{L^2}^2 \leq \bar C \big( (\Re \lambda) \eta |u|^2_{L^2} + \int |U_x||u|^2 \, dx \big),
\end{equation}
for some $\bar C > 0$ and $\eta > 0$ with $\epsilon^2 / \eta$ sufficiently small.
\end{lemma}
\begin{proof}
Denote $K = K(U(x))$, and take the real part of the $L^2$ product of $Ku_x$ against \eqref{k1}. Since $K$ is skew-symmetric, the result is
\begin{equation}
\label{k27}
\Re \langle u_x, KA u_x \rangle = \Re (\lambda \langle Ku_x, u \rangle ) + \Re \langle Ku_x, Lq_x \rangle .
\end{equation}

Noticing also that $\Im \langle Ku_x, u\rangle = -\half \langle K_x u, u \rangle$, we obtain the bound
\begin{equation}
\label{k28}
\Re (\lambda \langle Ku_x, u \rangle ) \leq C (\Re \lambda) \big( \eta^{-1} |u_x|^2_{L^2} + \eta |u|_{L^2}^2 \big) + C |\Im \lambda| \int |U_x||u|^2 \, dx,
\end{equation}
for any $\eta > 0$ and some $C>0$. We also have the estimate
\begin{equation}
\label{k29bis}
\langle Ku_x, Lq_x \rangle \leq C \big( \delta_1 |u_x|_{L^2}^2 + \delta_1^{-1} |q_x|_{L^2}^2 \big),
\end{equation}
for any $\delta_1 > 0$, where we have used Young's inequality in both estimates.

To estimate $\Re \langle u_x, KA u_x \rangle$, observe that from \eqref{KALB}, there holds
\begin{equation}
\label{k30}
\Re \langle u_x, KA u_x \rangle + \langle u_x, \bL B u_x \rangle \geq c^{-1}|u_x|_{L^2}^2,
\end{equation}
for some $c > 0$. (Notice that $\langle u_x, \bL B u_x \rangle \in \R$ because $\bL B$ is symmetric, positive semi-definite.)

Multiply equation \eqref{k2} by $\bL$, take the $L^2$ product with $u_x$ and integarte by parts. This yields,
\begin{equation}
\label{k34}
\langle u_x, \bL B u_x \rangle = -\langle u_{xx}, \bL q_x \rangle - \langle u_x, \bL_x q_x \rangle - \langle u_x, \bL q \rangle.
\end{equation}

To estimate the first term, take the real part of the $L^2$ product of $u_{xx}$ against $A_0$ times \eqref{k1}, use $\bA$ symmetric, $A_0$ positive definite, and integrate by parts to obtain
\begin{equation}
\label{labuena}
\begin{aligned}
-\Re \langle u_{xx}, \bL q_x \rangle &\leq - \Re ( \lambda \langle u_x, (A_0)_x u \rangle ) + \half \langle u_x, \bA_x u_x \rangle - \Re \langle u_x, \bA_x u \rangle \\ &\leq -(\Re \lambda) \Re \langle u_x, (A_0)_x u \rangle + (\Im \lambda)\Im \langle u_x, (A_0)_x u \rangle  + \\ &\quad + \half \langle u_x, \bA_x u_x \rangle - \Re \langle u_x, \bA_x u \rangle.
\end{aligned}
\end{equation}
Using \eqref{friedRe} and \eqref{friedIm}, and bounding the error terms $(A_0)_x, \bA_x = \cO(|U_x|) = \cO(\epsilon^2)$, we get
\begin{equation}
\label{k35}
-\Re \langle u_{xx}, \bL q_x \rangle \leq C\epsilon \int |U_x||u|^2 \, dx + C\epsilon |u_x|_{L^2}^2,
\end{equation}
where the term $\langle u_x, \bA_x u \rangle$ has been bounded by
\[
\begin{aligned}
\int |U_x||u||u_x| \, dx  &\leq \frac{C}{2} \Big( \int |U_x|^{3/2}|u|^2 \, dx + \int |U_x|^{1/2} |u_x|^2 \, dx \Big)\\
&\leq \frac{C}{2} \epsilon \int |U_x||u|^2 \, dx + \frac{C}{2} \epsilon |u_x|^2_{L^2}.
\end{aligned}
\]

We also estimate
\begin{align}
\Re \langle u_x, \bL_x q_x \rangle &\leq C \epsilon^2 |u_x|^2_{L^2} + C |q_x|_{L^2}^2, \\
\Re \langle u_x, \bL q \rangle &\leq C \big( \delta_2 |u_x|^2_{L^2} + \delta_2^{-1} |q|^2_{L^2} \big),
\end{align}
for any $\delta_2 > 0$, using Young's inequality. Putting all together back into \eqref{k34} we get
\begin{equation}
\label{k38}
\langle u_x, \bL B u_x \rangle \leq C \epsilon |u_x|_{L^2}^2 + C \int |U_x||u|^2 \, dx,
\end{equation}
after using \eqref{friedRe}.

Finally, since $\Re \lambda = \cO(\epsilon^2)$, taking $\delta_2 = \epsilon$ and $\epsilon^2 / \eta$ sufficiently small, we can substitute \eqref{k38}, \eqref{k28} and \eqref{k29bis} back into \eqref{k30}, absorb the small terms into the left hand side to obtain \eqref{k39}. This proves the result.
\end{proof}

\begin{corollary}
For all $\epsilon > 0$ sufficiently small and $\Re \lambda \geq 0$, there holds the estimate
\begin{equation}
\label{k41}
(\Re \lambda) |u|_{L^2}^2 + |u_x|_{L^2}^2 \leq C \int |U_x||u|^2 \, dx,
\end{equation}
for some $C > 0$.
\end{corollary}
\begin{proof}
Take $\bar C$ times estimate \eqref{friedRe} and add to \eqref{k39} to obtain
\[
\bar C(\Re \lambda) |u|_{L^2}^2 + |u_x|_{L^2}^2 \leq \bar C(1+C) \int |u_x||u|^2 \, dx + \bar C (\Re \lambda) \eta |u|_{L^2}^2.
\]
Take $\eta$ sufficiently small, say $\eta = \cO(\epsilon)$ so that $\epsilon^2 /\eta$ remains small, and after absorbing into the left hand side  we obtain the result.
\end{proof}

\subsection{Goodman-type estimate}

Finally, we control the term $\int |U_x||u|^2$ by performing a weighted energy estimate in the spirit of Goodman \cite{Go1,Go3} (see also \cite{HuZ1,MaZ6}).

\begin{lemma}\label{lem-Goodman}
Under \eqref{S0} - \eqref{S2}, \eqref{H0} - \eqref{H3}, for all $\Re \lambda \geq 0$ there holds the estimate
\begin{equation}
\label{g30}
\Re \lambda \Big(|u|^2_{L^2} +|u_x|^2_{L^2}\Big)+ \hat C \int |U_x||u|^2 \, dx \leq \hat C \epsilon |u_x|^2_{L^2},
\end{equation}
for some $\hat C>0$ and all $\epsilon > 0$ sufficiently small.
\end{lemma}


We first recall that there are matrices $L_p,R_p$ diagonalizing
matrix $A$ such that
\begin{equation}
\label{diag-a}
\tilde A := L_p A R_p = \begin{pmatrix} A_1^- &&0\\&a_p&\\
0&&A_2^+\end{pmatrix}
\end{equation}
where $A_\pm$ are symmetric and
positive/negative definite, and $a_p$ is scalar satisfying \eqref{monotonicity}
and $a_p = \cO(\epsilon)$.
%
Defining $v:=L_pu$, we rewrite \eqref{evalsyst} as
 \begin{equation}
     \begin{aligned}
     \lambda v + \tilde A v_x +\tilde Lq_{x} &= \tilde A (L_p)_x R_p v,\\
 -q_{xx} + q +\tilde B v_{x} &=-B (R_p)_x v,
    \end{aligned}
    \label{e-eqsv}
\end{equation}
where $$\tilde A = L_pAR_p, \qquad\tilde L = L_pL, \qquad\tilde B =
BR_p.$$
Define \begin{equation}\label{weightS}S := \begin{pmatrix}\phi_-
I_{p-1}&&0\\&1&\\0&&\phi_+ I_{n-p}\end{pmatrix}\end{equation} where block
diagonal form is in the same way as of \eqref{diag-a} and $\phi_\pm
$ are scalar functions of $x \in \R$ which are bounded away from zero and
satisfying
$$\phi'_\pm = \mp c_*|U_x|\phi_\pm, \qquad \phi_\pm (0)=1$$ for some
sufficiently large constant $c_*$ to be determined later.
Once again, we alternatively write $'$ or $(\cdot)_x$ as derivative with respect to $x$.

In what follows, we shall use $\iprod{\cdot,\cdot}$ as a weighted
norm defined by
$$\iprod{f,f} := \iprod{Sf,f}_{L^2}.$$
With this inner product, we note that for any symmetric matrix $A$,
$$\iprod{Af_x,f} = -\frac 12 \iprod{(A_x + (S_x/S)A)f,f}$$ where
$S_x/S$ should be understood as $\phi'_\pm/\phi_\pm$ or $0$
in each corresponding block.

By our choice of $S$ and $\phi_\pm$, we observe that
\begin{equation}\label{anest-S}\begin{aligned}\tilde A_x + (S_x/S)\tilde A &= \begin{pmatrix}(A_1^-)'
+(\phi'_-/\phi_-)A_1^-&&0\\&a'_p&\\0&&(A_2^+)'
+(\phi'_+/\phi_+)A_2^+\end{pmatrix}
\\&\le \begin{pmatrix}-c_*I&&0\\&-\theta&\\0&&-c_*I\end{pmatrix}|U_x|\end{aligned}\end{equation}

\begin{proposition}\label{prop-0thest} Denoting $v =: (v_-,v_p,v_+)^\top$, we obtain
\begin{equation}\label{goodineq1}\begin{aligned}(\Re \lambda) \iprod{v,v}& + \frac
12c_*\iprod{|U_x|v_\pm,v_\pm} +\frac 12\theta\iprod{|U_x|v_p,v_p} \le -\Re\iprod{\tilde
Lq_x,v}.\end{aligned}\end{equation}
\end{proposition}
\begin{proof}
We take inner product in the weighted norm of the first equation of
\eqref{e-eqsv} against $v$, take the real part of the resulting
equation, and make use of integration by parts, yielding
\begin{equation}\label{key1}\begin{aligned}(\Re \lambda) \iprod{v,v}& - \iprod{(\tilde A_x +
(S_x/S)\tilde A)_xv,v} =-\Re\iprod{\tilde Lq_x,v} + \Re\iprod{\tilde A
(L_p' R_p v,v}.\end{aligned}\end{equation}
Noting that $L_p'R_p  = \cO(|U_x|)$ and the fact that
$\tilde A$ has the diagonal block \eqref{diag-a}, we estimate
$$|\iprod{\tilde A L_p' R_p v,v}| \le C\iprod{|U_x|v_\pm,v_\pm} +C\iprod{|a_p||U_x|v_p,v_p}.$$

Using this, \eqref{anest-S} and the fact that $|a_p| =
\cO(\epsilon)$ is sufficiently small and $c_*$ is sufficiently
large, \eqref{key1} immediately yields \eqref{goodineq1}.
\end{proof}

\begin{proposition}\label{prop-1stest} We obtain
\begin{equation}\label{goodineq2}(\Re\lambda) \iprod{v_x,v_x} -\Re\iprod{\tilde
Lq_x,v}
\le C\iprod{\cO(|U_x|^2)v,v}+\eta\iprod{v_x,v_x}\end{equation}
for sufficiently small  $\eta>0$.\end{proposition}
\begin{proof}
We now take the inner product of the derivative of the first
equation of \eqref{e-eqsv} against $v_x$. We thus obtain
\begin{equation}\label{key2}\lambda \iprod{v_x,v_x} + \iprod{(\tilde A v_x)_x,v_x} +\iprod{(\tilde Lq_x)_x,v_x}
= \iprod{(\tilde A
L_p'R_p v)_x,v_x}\end{equation}
where we estimate by integration by parts,
$$\begin{aligned} \iprod{(\tilde A v_x)_x,v_x} &= \iprod{\tilde A_x
v_x,v_x} -\frac12 \iprod{(\tilde A_x +(S_x/S)\tilde A)v_x,v_x} = \iprod{\cO(|U_x|) v_x,v_x}\\\iprod{(\tilde A L_p'R_p v)_x,v_x} &= \iprod{\cO(|U_x|) v_x,v_x} + \iprod{\cO(|U_x|) v,v_x}
\end{aligned}$$
and by using the second equation and the semi-definite condition $\tilde L\tilde B \ge 0$,
$$\begin{aligned} \iprod{(\tilde Lq_x)_x,v_x} &= \iprod{\tilde Lq_{xx},v_x}+\iprod{\tilde
L_xq_x,v_x}\\
&= \iprod{\tilde L (q+\tilde B v_x+ BR_p' v),v_x}+\iprod{\tilde
L_xq_x,v_x}
\\
&= -\iprod{\tilde Lq_x,v}-\iprod{(\tilde L_x+(S_x/S)\tilde L)q,v}+\iprod{\tilde
L\tilde B v_x,v_x}+\iprod{\tilde L BR_p' v,v_x}+\iprod{\tilde
L_xq_x,v_x}
\\
&\ge -\iprod{\tilde Lq_x,v}+\iprod{\tilde
LBR_p' v,v_x}+\iprod{\tilde L_xq_x,v_x} -\iprod{(\tilde L_x+(S_x/S)\tilde L)q,v}.
\end{aligned}$$

Thus, \eqref{key2} yields
\begin{equation}\label{key3}\begin{aligned}(\Re\lambda) &\iprod{v_x,v_x} -\Re\iprod{\tilde Lq_x,v}
\\&\le \iprod{\cO(|U_x|^2) v,v}+\eta\iprod{v_x,v_x}+\iprod{\tilde L_xq_x,v_x}
-\iprod{(\tilde L_x+(S_x/S)\tilde L)q,v}.\end{aligned}\end{equation}

By testing the second equation against $q$, it is easy to see that
$$\iprod{q_x,q_x} + \iprod{q,q} \le C\iprod{v_x,v_x}.$$
Thus, we have $$\iprod{\tilde L_xq_x,v_x} -\iprod{(\tilde L_x+(S_x/S)\tilde L)q,v}\le
C\iprod{v_x,v_x}^{1/2}\Big(\iprod{\cO(|U_x|^2)v_x,v_x}+\iprod{\cO(|U_x|^2)v,v}\Big)^{1/2}
$$

Using the standard Young's inequality and absorbing all necessary
terms into the right hand side of \eqref{key3}, we thus obtain from
\eqref{key3} the important estimate, \eqref{goodineq2}, which proves
the proposition.
\end{proof}

Combining Propositions \ref{prop-0thest} and \ref{prop-1stest}, we
are now ready to give:
\begin{proof}[Proof of Lemma \ref{lem-Goodman}]
Adding \eqref{goodineq2} with \eqref{goodineq1}, noting that the
``bad'' term $\Re\iprod{\tilde Lq_x,v}$ gets canceled out, and using
the fact that $|U_x| = \cO(\epsilon)$ is sufficiently small, we easily obtain
\begin{equation}\label{finalineq}\begin{aligned}\Re \lambda (\iprod{v,v}+\iprod{v_x,v_x})&
+ \theta\iprod{|U_x|v,v} \le \eta \iprod{v_x,v_x}\end{aligned}\end{equation}
which by changing $v$ to the original coordinate $u$ yields the lemma.
\end{proof}

\subsection*{Proof of Theorem \ref{spectralstability}}

Add $\hat C\epsilon$ times \eqref{k41} to \eqref{g30} to get
\[
(\Re \lambda)(1 + \hat C\epsilon) |u|^2_{L^2} + (\hat C + C\hat C\epsilon) \int |U_x||u|^2 \, dx \leq 0,
\]
which readily implies $\Re \lambda < 0$, yielding the result.
\qed

\begin{remark}
 \label{nonconvexspect}
Theorem \ref{spectralstability} can be extended to the non-convex case, that is, when the principal characteristic mode is no longer genuinely nonlinear (hypothesis \eqref{H2} does not hold). For that purpose, it is possible to modify the Goodman-type weighted energy estimate by means of the Matsumura-Nishihara weight function $w$ \cite{MN2} (introduced to compensate for the loss of monotonicity), satisfying
\[
 -\half(w a_p + w_x) = |U_x|,
\]
which replaces the $1$ in the weight matrix function $S$ in \eqref{weightS}. This procedure was carried out for the viscous systems case by Fries \cite{Fri1} and it can be done in the present case as well at the expense of further book-keeping. Note that the existence result of \cite{LMS1,LMS2} includes non-convex systems, a feature that might be useful in applications.
\end{remark}

\section{Pointwise reduction lemma}\label{firstAppendix}
Let us consider the situation of a system of equations of form
\begin{equation} \label{eq:firstorder}
	W_x = \A^\epsilon(x,\lambda)W,
\end{equation}
for which the coefficient $\A^\epsilon$ does not exhibit uniform
exponential decay to its asymptotic limits, but instead is {\it
slowly varying} (uniformly on a $\epsilon$-neighborhood $\cV$, being
$\epsilon>0$ a parameter).
This case occurs in different contexts for rescaled equations, such
as \eqref{eq:strechtedsyst} in the present analysis.

In this situation, it frequently occurs that not only $\A^\epsilon$
but also certain of its invariant eigenspaces are slowly varying
with $x$, i.e., there exist matrices
\begin{equation*}
	{\mathbb{L}}^\epsilon=\begin{pmatrix} L^\epsilon_1 \\
	L^\epsilon_2\end{pmatrix}(x), \quad \R^\epsilon=\begin{pmatrix}
	R^\epsilon_1 & R^\epsilon_2\end{pmatrix} (x) \label{LR}
\end{equation*}
for which ${\mathbb{L}}^\epsilon \R^\epsilon(x)\equiv I$ and
$|{\mathbb{L}}\R'|=|{\mathbb{L}}'\R|\le C\delta^\epsilon(x)$, uniformly in $\epsilon$,
where the pointwise error bound $\delta^\epsilon =
\delta^\epsilon(x)$ is small, relative to
\begin{equation}\label{M}
	{\mathbb{M}}^\epsilon:= {\mathbb{L}}^\epsilon \A^\epsilon \R^\epsilon(x)
	=\begin{pmatrix} M^\epsilon_1 & 0 \\
		0 & M^\epsilon_2 \end{pmatrix}(x)
\end{equation}
and ``$'$'' as usual denotes $\partial/\partial x$.
In this case, making the change of coordinates $W^\epsilon=\R^\epsilon Z$, we
may reduce \eqref{eq:firstorder} to the approximately block-diagonal equation
\begin{equation}\label{eq:blockdiag}
	{Z^\epsilon}'= {\mathbb{M}}^\epsilon Z^\epsilon + \delta^\epsilon
	\Theta^\epsilon Z^\epsilon,
\end{equation}
where ${\mathbb{M}}^\epsilon$ is as in \eqref{M}, $\Theta^\epsilon(x)$ is a
uniformly bounded matrix, and $\delta^\epsilon(x)$ is (relatively) small.
Assume that such a procedure has been successfully carried
out, and, moreover, that there exists an approximate {\it uniform
spectral gap in numerical range}, in the strong sense that
\begin{equation*} \label{eq:gap}
	\min \sigma(\Real M_1^\epsilon)- \max \sigma(\Real M_2^\epsilon)
	\ge \eta^\epsilon(x), \qquad \textrm{\rm for all } x,
\end{equation*}
with pointwise gap $\eta^\epsilon(x) > \eta_0 > 0$
uniformly bounded in $x$ and in $\epsilon$; here and elsewhere
$\Real N:= \half (N+N^*)$ denotes the ``real'', or symmetric part of an operator $N$.
Then, there holds the following {\it pointwise reduction lemma}, a refinement of
the reduction lemma of \cite{MaZ3} (see the related ``tracking lemma'' given in
varying degrees of generality in \cite{GZ,MaZ1,PZ,ZH,Z3}).
\medskip

\begin{proposition}\label{pwrl}
Consider a system \eqref{eq:blockdiag} under the gap
assumption \eqref{eq:gap}, with $\Theta^\epsilon$ uniformly bounded
in $\epsilon \in \cV$ and for all $x$.  If, for all $\epsilon \in
\cV$, $\sup_{x \in \R} (\delta^\epsilon/\eta^\epsilon)$ is
sufficiently small (i.e., the ratio of pointwise gap
$\eta^\epsilon(x)$ and pointwise error bound $\delta^\epsilon(x)$ is
uniformly small), then there exist (unique) linear transformations
$\Phi_1^\epsilon(x,\lambda)$ and $\Phi_2^\epsilon(x,\lambda)$,
possessing the same regularity with respect to the various
parameters $\epsilon$, $x$, $\lambda$ as do coefficients
${\mathbb{M}}^\epsilon$ and $\delta^\epsilon(x)\Theta^\epsilon(x)$, for which
the graphs $\{(Z_1, \Phi^\epsilon_2 (Z_1))\}$ and
$\{(\Phi^\epsilon_1(Z_2),Z_2)\}$ are invariant under the flow of
\eqref{eq:blockdiag}, and satisfying
\begin{equation*}
	\sup_\R |\Phi^\epsilon_j| \, \le \, C \sup_\R (\delta^\epsilon/\eta^\epsilon).
\end{equation*}
Moreover, we have the pointwise bounds
\begin{equation}\label{ptwise}
	|\Phi^\epsilon_2(x)|\le C \int_{-\infty}^x e^{-\int_y^x
	\eta^\epsilon(z)dz} \delta^\epsilon(y) dy,
\end{equation}
and symmetrically for $\Phi^\epsilon_1$.
\end{proposition}
\medskip

\begin{proof}
By a change of independent coordinates, we may arrange that $\eta^\epsilon(x)\equiv$
constant, whereupon the first assertion reduces to the conclusion of the tracking/reduction
lemma of \cite{MaZ3}.
Recall that this conclusion was obtained by seeking $\Phi^\epsilon_2$ as the solution
of a fixed-point equation
\begin{equation*}
	\Phi^\epsilon_2(x)= {\mathcal T }\Phi^\epsilon_2(x):= \int_{-\infty}^x
	\cF^{y\to x} \delta^\epsilon(y) Q(\Phi^\epsilon_2)(y) dy.
\end{equation*}
Observe that in the present context we have allowed $\delta^\epsilon$ to vary with $x$,
but otherwise follow the proof of \cite{MaZ3} word for word to obtain the conclusion
(see Appendix C of \cite{MaZ3}, proof of Proposition 3.9).
Here, $Q(\Phi^\epsilon_2)=\cO(1+|\Phi^\epsilon_2|^2)$ by construction, and
$|\cF^{y\to x}|\le Ce^{-\eta(x-y)}$.
Thus, using only the fact that $|\Phi^\epsilon_2|$ is bounded, we obtain the
bound \eqref{ptwise} as claimed, in the new coordinates for which $\eta^\epsilon$
is constant.
Switching back to the old coordinates, we have instead
$|\cF^{y\to x}|\le Ce^{-\int_y^x \eta^\epsilon(z)dz}$, yielding the result in the general case.
\end{proof}
\medskip

\begin{remark}\label{rem:reduced}\rm
From Proposition \ref{pwrl}, we obtain reduced flows
\begin{equation*}
	\left\{\begin{aligned}
		{Z_1^\epsilon}' &= M_1^\epsilon Z_1^\epsilon + \delta^\epsilon(
			\Theta_{11} + \Theta_{12}^\epsilon \Phi_2^\epsilon) Z_1^\epsilon,\\
		{Z_2^\epsilon}' &= M_2^\epsilon Z_2^\epsilon +
			\delta^\epsilon(\Theta_{22}+ \Theta_{21}^\epsilon \Phi_1^\epsilon)
			Z_2^\epsilon.
	\end{aligned}\right.
\end{equation*}
on the two invariant manifolds described.
\end{remark}

\end{document}